\documentclass[a4paper,12pt,leqno]{amsart}
\usepackage{latexsym}
\usepackage[all]{xy}
\usepackage{amssymb} 
\usepackage{amsmath} 
\usepackage{amscd}
\usepackage{color}
\usepackage{comment}
\usepackage{mathtools}
\usepackage{fancybox}
\usepackage{pb-diagram}
\usepackage{bm}

\numberwithin{equation}{section} 

\newtheorem{theorem}{Theorem}[section]
\newtheorem{lemma}[theorem]{Lemma} 
\newtheorem{corollary}[theorem]{Corollary}
\newtheorem{proposition}[theorem]{Proposition} 
 
\newtheorem{remark}[theorem]{Remark}

\newtheorem{example}{Example}
\newtheorem{definition}[theorem]{Definition}

\allowdisplaybreaks[3]

\def\C{\mathbb C}
\def\R{\mathbb R}

\def\Z{\mathbb Z}

\DeclareMathOperator{\codim}{codim}

\DeclareMathOperator{\Int}{Int}

\newcommand{\onto}{\twoheadrightarrow}

\newcommand{\ctc}[2]{\tau_{#1,#2}}
\newcommand{\fan}{\Sigma}
\newcommand{\nfan}{\Sigma'}
\newcommand{\Sn}{\mathfrak{S}_n}

\newcommand{\FP}[1]{F^{+}_{#1}}
\newcommand{\FM}[1]{F^{-}_{#1}}
\newcommand{\FPM}[1]{F^{\pm}_{#1}}
\newcommand{\FCE}[2]{F_{#1,#2}}
\newcommand{\Ad}[1]{\text{\rm Ad}_{#1}}
\newcommand{\G}[1]{{SL_{#1}}}
\newcommand{\Jb}[1]{{J_{#1}^*}}
\newcommand{\RS}{Y}
\newcommand{\EL}{E}
\newcommand{\yn}{n-1}
\newcommand{\dyn}{[n-1]}
\newcommand{\ddyn}{2(n-1)}
\newcommand{\Xp}[1]{X(\fan)_{{#1};> 0}}
\newcommand{\Xpp}[1]{X'_{{#1};> 0}}
\newcommand{\ic}{k}

\begin{document}

\title[Totally nonnegative part of the Peterson variety]{Totally nonnegative part of the Peterson variety\\ in Lie type A}
\author {Hiraku Abe}
\address{Faculty of Science, Department of Applied Mathematics, Okayama University of
Science, 1-1 Ridai-cho, Kita-ku, Okayama, 700-0005, Japan}
\email{hirakuabe@globe.ocn.ne.jp}

\author {Haozhi Zeng}
\address{School of Mathematics and Statistics, Huazhong University of Science and Technology, Wuhan, 430074, P.R. China}
\email{zenghaozhi@icloud.com} 

\begin{abstract}
The Peterson variety (which we denote by $Y$) is a subvariety of the flag variety, introduced by Dale Peterson to describe the quantum cohomology rings of all the partial flag varieties.
Motivated by the mirror symmetry for partial flag varieties, Rietsch studied the totally nonnegative part $Y_{\ge0}$ and its cell decomposition.
Based on the structure of those cells, Rietsch gave the following conjecture in Lie type A; as a cell decomposed space, $Y_{\ge0}$ is homeomorphic to the cube $[0,1]^{\dim_{\C}Y}$.
In this paper, we give a proof of Rietsch's conjecture on $Y_{\ge0}$ in Lie type A by using toric geometry which is closely related to the Peterson variety.
\end{abstract}

\maketitle

\section{Introduction}\label{sect: intro}
An invertible matrix is said to be totally positive (resp.\ nonnegative) if all its minors are positive (resp.\ nonnegative).
In 1990's, Lusztig extended in \cite{lu94} the theory of totally positive matrices to an arbitrary split reductive connected algebraic group over $\R$. It was motivated by Kostant's question on the relation between totally positive matrices and positivity properties of canonical bases for quantized enveloping algebras. He also introduced the totally nonnegative part of the flag varieties in \cite{lu94,lu98}, and now total positivity appears in several fields of mathematics including cluster algebras (\cite{fo10,fo-ze02}), KP equations (\cite{ko-wi11,ko-wi14}), and mathematical physics (\cite{ga-py20,li17,ri12}).
For general properties of the totally nonnegative parts of the flag varieties, see e.g. \cite{ga-ka-lam22,he14,lu94,po-sp-wi09,ri99,ri-wi08}.

Also in 1990's, Dale Peterson (\cite{pe-97}) observed that there is an algebraic variety $Y$ endowed with a certain stratification whose strata describe the quantum cohomology rings of all the partial flag varieties.
This variety $Y$ is now called the Peterson variety, and in Lie type A, it can be defined as follows:
\begin{equation*}
Y\coloneqq \{gB^-\in \G{n}(\mathbb{C})/B^- \mid (\Ad{g^{-1}} f)_{i,j}=0\ (j>i+1)\}, 
\end{equation*}
where $B^-\subseteq \G{n}(\C)$ is the subgroup consisting of lower triangular matrices, and $f$ is the $n\times n$ regular nilpotent matrix in Jordan canonical form.
In \cite{ri03}, Rietsch gave a proof of Peterson's claim for the quantum cohomology ring $qH^*(\G{n}(\mathbb{C})/P)$ for parabolic subgroups $P\subseteq \G{n}(\mathbb{C})$ (cf.\ \cite{ch22}). By combining it with Lusztig's theory of total positivity on $\G{n}(\mathbb{C})/B^-$, she obtained a parametrization of the set of $n\times n$ totally nonnegative Toeplitz matrices. Here, Toeplitz matrices are the matrices of the form
\begin{equation*}
\begin{pmatrix} 
  1&  &&&\\
  x_1 & \!\!\!1 &  && \\
  x_2 & \!\!\!x_1 & \!\!\!1 &  &\\
  \vdots & \!\!\!x_2 & \!\!\!x_1 & \ddots\\
  x_{n-2} & \!\!\!\vdots & \!\!\!\ddots & \ddots & \ddots\\
  x_{n-1} & \!\!\!x_{n-2}& \!\!\!\cdots & x_2 & x_1& 1
\end{pmatrix}
\end{equation*}
for $x_1,x_2,\ldots,x_{n-1}\in \C$.
In fact, Rietsch showed that the totally nonnegative Toeplitz matrices are completely parametrized by the lower left $i\times i$ minors for $1\le i\le n-1$ (\cite[Proposition~1.2]{ri03}). 
One of the key ingredients of the proof was to introduce the totally nonnegative parts of the strata of the Peterson variety $Y$ in the connection to Peterson's description of quantum cohomology rings $qH^*(\G{n}(\mathbb{C})/P)$ of partial flag varieties.
Motivated by the mirror symmetry for $\G{n}(\mathbb{C})/P$, Rietsch studied a cell decomposition of the totally nonnegative part $Y_{\ge0}$ (\cite{ri03,ri06}) whose cells are given by the intersections with open Richardson varieties, and 
she also showed that the whole space $Y_{\ge0}$ is contractible.
Based on the observations on the structure of the cells, Rietsch gave the following conjecture (\cite[Conjecture~10.3]{ri06}), where we set $[0,1]=\{x\in\R\mid 0\le x\le 1\}$.\\

\noindent
\textbf{Rietsch's conjecture on $Y_{\ge0}$}\textbf{.}
As a cell decomposed space, $Y_{\ge 0}$ is homeomorphic to the $(n-1)$-dimensional cube $[0,1]^{\yn}$. \\

In this paper, we prove that Rietsch's conjecture on $Y_{\ge0}$ holds. The idea of our proof is as follows. In \cite{ab-ze23}, the authors constructed a particular morphism 
$\Psi \colon Y\rightarrow X(\fan)$, 
where $X(\fan)$ is the toric orbifold introduced by M. Blume (\cite{bl15}).
We show that this morphism restricts to a homeomorphism on their nonnegative parts:
\begin{equation*}
\Psi_{\ge0} \colon 
Y_{\ge0} \stackrel{\approx}{\rightarrow} X(\fan)_{\ge0}.
\end{equation*}
The toric orbifold $X(\fan)$ is projective so that there is a simplicial lattice polytope $P$ whose normal fan is $\fan$. This polytope $P$ is combinatorially equivalent to an $(n-1)$-dimensional cube $[0,1]^{\yn}$ which implies that we have $P\approx [0,1]^{\yn}$. Since $X(\fan)_{\ge0}$ is homeomorphic to the polytope $P$ under the moment map, we conclude that 
\begin{equation*}
Y_{\ge0} \approx [0,1]^{\yn}.
\end{equation*}
We also show that the resulting homeomorphism preserves the cells of $Y$ and $[0,1]^{\yn}$.
As one can see, our proof uses toric geometry. A similar approach was taken in \cite{po-sp-wi09,ri-wi08} to show that the nonnegative parts of partial flag varieties are CW complexes.

This paper is organized as follows. In Section 2, we fix some notations which we use throughout this paper. In Sections 3 and 4, we describe the toric orbifold $X(\Sigma)$ and its nonnegative part $X(\Sigma)_{\ge 0}$. In Sections 5 and 6, we study the Peterson variety $Y$ and its totally nonnegative part $Y_{\ge 0}$.
In Sections 7 and 8, we recall the definition of the morphism 
$\Psi \colon Y \rightarrow X(\fan)$, and we show that $Y_{\ge 0}$ and $X(\Sigma)_{\ge 0}$ are homeomorphic as cell decomposed spaces through $\Psi$. Finally, in Section 9, we give a proof of Rietsch's conjecture on $Y_{\ge0}$.

\vspace{20pt}

\noindent \textbf{Acknowledgments}.
We are grateful to Konstanze Rietsch, Bin Zhang, Changzheng Li, Naoki Fujita, Mikiya Masuda, Takashi Sato, and Tatsuya Horiguchi for valuable discussions. This research is supported in part by Osaka Central Advanced Mathematical Institute (MEXT Joint Usage/Research Center on Mathematics and Theoretical Physics): Geometry and combinatorics of Hessenberg varieties. The first author is supported in part by JSPS Grant-in-Aid for Scientific Research(C): 23K03102. The second author is supported in part by NSFC: 11901218.

\vspace{20pt}

\section{Notations}\label{sec: notations}
We fix some notations which we use throughout this paper. 
Let $n\ge 2$ be a positive integer, and for simplicity let us denote by
\begin{align*}
 \G{n}\coloneqq \G{n}(\C)
\end{align*}
the complex special linear group of degree $n$.
Let $B^+$ (resp. $B^-$) be the Borel subgroup of $\G{n}$ consisting of upper (resp. lower) triangular matrices.
We denote by $T\coloneqq B^+\cap B^-$ the maximal torus of $\G{n}$ consisting of diagonal matrices:
\begin{align*}
 T = \{\text{diag}(t_1,\ldots,t_n)\in \G{n} \mid t_i\in\C^{\times}\ (1\le i\le n)\}.
\end{align*}
Let $U^+$ (resp. $U^-$) be the unipotent subgroup of $\G{n}$ consisting of upper (resp. lower) triangular matrices having $1$'s on the diagonal.

We write $\dyn\coloneqq\{1,2,\ldots,n-1\}$.
For $i\in[n-1]$, let $\alpha_i \colon T\rightarrow \C^{\times}$ be the $i$-th simple root of $T$ given by
\begin{align*}
 \alpha_i(t) \coloneqq t_i t_{i+1}^{-1}
 \qquad \text{for $t=\text{diag}(t_1,\ldots,t_n)\in T$},
\end{align*}
and let $\varpi_i \colon T\rightarrow \C^{\times}$ be the $i$-th fundamental weight of $T$ given by
\begin{align*}
 \varpi_i(t) \coloneqq t_1t_2\cdots t_i
 \qquad \text{for $t=\text{diag}(t_1,\ldots,t_n)\in T$}.
\end{align*}

\vspace{20pt}

\section{Toric orbifolds associated to Cartan matrices}\label{sec: Toric orbifold}
In this section, we describe the projective toric orbifold $X(\fan)$ appeared in the introduction. We begin with its quotient description, and we describe its fan $\fan$. 
The goal of this section is to describe the nonnegative part $X(\fan)_{\ge0}$ with its cell decomposition.

\subsection{Definition by quotient}\label{subsec: Cartan toric orbifolds}
We take the following description of $\C^{\ddyn}$:
\begin{align*}
 \C^{\ddyn} =
 \{ (x_1,\ldots,x_{n-1};y_1,\ldots,y_{n-1}) \mid x_i,y_i\in\C \ (i\in [n-1]) \}.
\end{align*}
The maximal torus $T\subseteq \G{n}$ has the following linear action on $\C^{\ddyn}$: 
\begin{align}\label{eq: def of T-action on C2r}
\begin{split}
 &t\cdot (x_1,\ldots,x_{n-1};y_1,\ldots,y_{n-1}) \\
 &\hspace{20pt}\coloneqq\big(\varpi_1(t)x_1,\ldots,\varpi_{n-1}(t)x_{n-1}; \alpha_1(t)y_1,\ldots,\alpha_{n-1}(t)y_{n-1}\big)
\end{split}
\end{align}
for $t\in T$ and $(x_1,\ldots,x_{n-1};y_1,\ldots,y_{n-1})\in \C^{\ddyn}$.
Namely, $T$ acts on the first $n-1$ components by the fundamental weights and on the last $n-1$ components by the simple roots.
We define a subset $\EL\subset \C^{2(n-1)}$ by
\begin{align}\label{eq: def of exceptional locus}
 \EL \coloneqq \bigcup_{i\in \dyn} \{ (x_1,\ldots,x_{n-1}; y_1,\ldots,y_{n-1})\in\C^{2(n-1)} \mid  x_i=y_i=0 \}.
\end{align}
Then its complement is given by
\begin{align}\label{eq: def of C-E}
 \C^{\ddyn}-\EL =
 \{ (x_1,\ldots,x_{n-1};y_1,\ldots,y_{n-1}) \mid (x_i,y_i)\ne(0,0) \ (i\in [n-1]) \}.
\end{align}
It is obvious that the $T$-action on $\C^{\ddyn}$ defined in \eqref{eq: def of T-action on C2r} preserves this subset $\C^{\ddyn}-\EL$. We now consider the quotient 
\begin{align*}
 (\C^{\ddyn}-\EL)/T
\end{align*}
by the action of the maximal torus $T\subseteq \G{n}$ given in \eqref{eq: def of T-action on C2r}. 

The quotient space $(\C^{\ddyn}-\EL)/T$ admits an action of a complex torus defined as follows.
The $2(n-1)$-dimensional torus $(\C^{\times})^{\ddyn}$ acts on $\C^{\ddyn}-\EL$ by the coordinate-wise multiplication, and the above $T$-action on $\C^{\ddyn}-\EL$ is obtained through the homomorphism
\begin{align}\label{eq: injection from T}
 T \rightarrow (\C^{\times})^{\ddyn}
 \quad ; \quad 
 t \mapsto \big(\varpi_1(t),\ldots,\varpi_{n-1}(t); \alpha_1(t),\ldots,\alpha_{n-1}(t)\big).
\end{align}
The injectivity of this homomorphism implies that the quotient torus $(\C^{\times})^{\ddyn}/T$ is defined, and an action of the torus $(\C^{\times})^{\ddyn}/T$ on $(\C^{\ddyn}-\EL)/T$ is naturally induced\footnote{In \cite{ab-ze23}, the action of $(\C^{\times})^{\ddyn}/T$ is identified with an action of $T/Z$, where $Z$ is the center of $\G{n}$. See \cite[Sect.~3.3]{ab-ze23} for details.}. 

In \cite{ab-ze23}, the authors showed that the quotient space $(\C^{\ddyn}-\EL)/T$ with this torus action is a simplicial projective toric variety of complex dimension $n-1$. In fact, we constructed a simplicial projective fan $\fan$ in $\R^{n-1}$ such that the quotient construction of the associated toric orbifold $X(\fan)$ agrees with the quotient map $$\C^{\ddyn}-\EL\rightarrow (\C^{\ddyn}-\EL)/T.$$
We review the construction of this fan $\fan$ in the next subsection.

\vspace{10pt}

\subsection{The fan for $(\C^{\ddyn}-\EL)/T$}\label{subsec: Cartan toric orbifolds from fan}
In this subsection, let us review how we see that the quotient $(\C^{\ddyn}-\EL)/T$ with the action of $(\C^{\times})^{\ddyn}/T$  is the toric orbifold $X(\fan)$ associated with a simplicial projective fan $\fan$.
Details can be found in \cite[Sect.\ 3 and 4]{ab-ze23}. 

We begin with defining this fan $\fan$. For $i\in\dyn$, we set
\begin{align*}
 -\alpha^{\vee}_i \coloneqq \bm{e}_{i-1}-2\bm{e}_i+\bm{e}_{i+1} \in \R^{\yn}
\end{align*}
with the convention $\bm{e}_0=\bm{e}_n=0$, 
where each $\bm{e}_i$ is the $i$-th standard basis vector in $\R^{\yn}$.
For $K,L\subseteq \dyn$ such that $K\cap L=\emptyset$, we set
\begin{align*}
 \sigma_{K,L} \coloneqq \text{cone}(\{ -\alpha^{\vee}_i \mid i\in K \}\cup \{ \bm{e}_{i} \mid i \in L \}), 
\end{align*}
where we take $\sigma_{\emptyset,\emptyset}=\{\bm{0}\}$ as the convention.
The dimension of this cone is given by
\begin{align}\label{eq: dim of cone}
 \dim_{\R} \sigma_{K,L} =|K|+|L|
\end{align}
(\cite[Lemma~3.3]{ab-ze23}) which implies that $\sigma_{K,L}$ is a simplicial cone.
We define
\begin{align}\label{eq: def of fan 1}
 \fan \coloneqq \{ \sigma_{K,L} \subseteq \R^{\yn} \mid K,L\subseteq \dyn, \ K\cap L=\emptyset \}.
\end{align}
In \cite{ab-ze23}, it is proved that $\fan$ with the lattice $\Z^{n-1}=\oplus_{i=1}^n\Z \bm{e}_i$ is a simplicial projective fan in $\R^{n-1}$ (\cite[Corollary~3.6 and Proposition~4.4]{ab-ze23}). 

We denote by $X(\fan)$ the simplicial projective toric variety associated to the fan $\fan$.
Cox's quotient construction of $X(\fan)$ is given as follows.
Let $\fan(1)$ be the set of rays (i.e.\ 1 dimensional cones) of $\fan$.
Namely,
\begin{align*}
 \fan(1) = \{\text{cone}(-\alpha^{\vee}_i) \mid i\in \dyn \} \cup\{\text{cone}(\bm{e}_i) \mid i\in \dyn \}.
\end{align*}
A collection of rays in $\fan(1)$ is called a primitive collection 
when it is a minimal collection which does not span a cone in $\fan$.
Recalling the definition of cones $\sigma_{K,L}$ in $\fan$, it is clear that each primitive collection is given by $\{\text{cone}(-\alpha^{\vee}_i), \text{cone}(\bm{e}_i)\}$ for each $i\in \dyn$.
Therefore, the exceptional set $\EL$ in the affine space $\C^{\fan(1)}=\C^{2(n-1)}$ determined by these primitive collections agrees with the one given in \eqref{eq: def of exceptional locus}:
\begin{align*}
 \EL = \bigcup_{i\in \dyn} \{ (x_1,\ldots,x_{\yn}; y_1,\ldots,y_{\yn})\in\C^{2(n-1)} \mid  x_i=y_i=0 \}
\end{align*}
(see \cite[Proposition~5.1.6]{co-li-sc}).
Here, we regard the first $n-1$ components of $\C^{2(n-1)}$ to correspond to the rays $\text{cone}(-\alpha^{\vee}_i)$ in order, and we regard the last $n-1$ components of $\C^{2(n-1)}$ to correspond to the rays $\text{cone}(\bm{e}_i)$ in order.
Now we consider the surjective homomorphism 
\begin{align*}
 \rho \colon (\C^{\times})^{2(n-1)}\rightarrow (\C^{\times})^{n-1}
\end{align*}
determined by the ray vectors $-\alpha^{\vee}_i$ and $\bm{e}_i$ as follows: 
\begin{align*}
 \rho(u_1,\ldots,u_{n-1};v_1,\ldots,v_{n-1}) = ( u^{-\alpha^{\vee}_1} v^{\bm{e}_1}, \ldots, u^{-\alpha^{\vee}_{n-1}} v^{\bm{e}_{n-1}})
\end{align*}
for $(u_1,\ldots,u_{n-1};v_1,\ldots,v_{n-1})\in (\C^{\times})^{2(n-1)}$, 
where $u^{-\alpha^{\vee}_i}\coloneqq u_{i-1}u_i^{-2}u_{n+1}$ and $v^{\bm{e}_i}\coloneqq v_i$ for $1\le i\le n-1$ 
with the convention $u_0=u_n=1$.
To describe the kernel of $\rho$, we recall the injective homomorphism $T \rightarrow (\C^{\times})^{2(n-1)}$ given in \eqref{eq: injection from T}, where $T\subseteq \G{n}$ is the maximal torus.
By a straightforward calculation, one can verify that $\rho$ fits in the exact sequence
\begin{align}\label{eq: exact sequence of tori}
 1\rightarrow T \rightarrow (\C^{\times})^{2(n-1)} \stackrel{\rho}{\rightarrow} (\C^{\times})^{n-1} \rightarrow 1,
\end{align}
where the second map is the one defined in \eqref{eq: injection from T}.
The homomorphism $\rho$ in this sequence extends to a morphism
\begin{align}\label{eq: geometric quotient 1}
 \C^{2(n-1)}-\EL\rightarrow X(\fan)
\end{align}
which is equivariant with respect to $\rho$.

Since the fan $\fan$ is simplicial (\cite[Corollary~3.6]{ab-ze23}), we can describe the toric variety $X(\fan)$ by Cox's quotient construction (\cite[Theorem~5.1.11]{co-li-sc}) which is known to work with possibly non-primitive ray vectors\footnote{When $n\ge3$, the ray vectors $-\alpha^{\vee}_i$ and $\bm{e}_i$ are primitive for $1\le i\le n-1$. When $n=2$, the ray vector $-\alpha^{\vee}_1$ becomes $-2\bm{e}_1$ which is not primitive.} by \cite[Proposition 3.7]{bo-ch-sm05} (cf.\ \cite[Sect.~3.3]{ab-ze23}).
It claims that the morphism \eqref{eq: geometric quotient 1} is a geometric quotient for the $T$-action on $\C^{2(n-1)}-\EL$.
This implies that the morphism \eqref{eq: geometric quotient 1} induces a bijection 
\begin{align}\label{eq: geometric quotient 2}
 (\C^{2(n-1)}-\EL)/T \rightarrow X(\fan)
\end{align}
so that we may regard the quotient $(\C^{2(n-1)}-\EL)/T$ as an algebraic variety which is isomorphic to $X(\fan)$ under the map \eqref{eq: geometric quotient 2}.
By construction, this isomorphism is equivariant with respect to the group isomorphism of tori $(\C^{\times})^{2(n-1)}/T \stackrel{\cong}{\rightarrow} (\C^{\times})^{n-1}$ induced by \eqref{eq: exact sequence of tori}.
Recalling that $\C^{2(n-1)}-\EL$ was given in \eqref{eq: def of C-E}, we obtain
\begin{align*}
 X(\fan) \cong \{[x_1,\ldots,x_{n-1};y_1,\ldots,y_{n-1}] \mid x_i, y_i\in\C, (x_i,y_i)\ne(0,0) \ (i\in \dyn)\}.
\end{align*}
In the rest of this paper, we identify these two varieties, and we take the action of $(\C^{\times})^{2(n-1)}/T$ as the canonical torus action on $X(\fan)$.

\vspace{10pt}

As is well-known, the cones in the fan $\fan$ correspond to the torus orbits in $X(\fan)$ which gives us the orbit decomposition of $X(\fan)$.
In later sections, we will study a Richardson type decomposition of the Peterson variety (see \eqref{eq: def of richardson strata}).
To compare these decompositions, it is more suited to consider the cones $\sigma_{K,\dyn-J}$ (rather than $\sigma_{K,J}$). Here, we note that the cone $\sigma_{K,\dyn-J}$ is defined for $K\subseteq J\subseteq \dyn$ since the condition $K\subseteq J$ is equivalent to $K\cap(\dyn-J)=\emptyset$.
Because of this reason, 
let us use the following notation:
\begin{align}\label{eq: def of cone 2}
 \ctc{K}{J} \coloneqq \sigma_{K,\dyn-J}
 =\text{cone}(\{ -\alpha^{\vee}_i \mid i\in K \}\cup \{ \bm{e}_{i} \mid i \notin J \})
\end{align}
for $K\subseteq J\subseteq \dyn$.
Since $\sigma_{K,\dyn-J}$ is a cone of dimension $|K|+(n-1)-|J|$ by \eqref{eq: dim of cone}, we obtain 
\begin{align*}
 \codim_{\R}\ctc{K}{J} = |J| - |K|.
\end{align*}
By the definition \eqref{eq: def of fan 1} of $\fan$, it can be expressed as
\begin{align}\label{eq: def of fan 2}
 \fan = \{ \ctc{K}{J} \subseteq \R^{\yn} \mid K\subseteq J\subseteq \dyn \}.
\end{align}

\vspace{5pt}

\begin{example}\label{ex: fan n=3}
{\rm When $n=3$, then $\fan$ is a fan in $\R^2$ which we depict in Figure~\ref{pic: fan2 and fan3}.
\begin{figure}[htbp]
\hspace{-20pt}
{\unitlength 0.1in%
\begin{picture}(23.7500,21.2500)(7.7000,-32.9500)%
%
\special{pn 0}%
\special{sh 0.200}%
\special{pa 2112 2227}%
\special{pa 2112 1181}%
\special{pa 3134 1181}%
\special{pa 3134 2227}%
\special{pa 2112 2227}%
\special{ip}%
\special{pn 8}%
\special{pa 2112 2227}%
\special{pa 2112 1181}%
\special{pa 3134 1181}%
\special{pa 3134 2227}%
\special{pa 2112 2227}%
\special{ip}%
%
\special{pn 0}%
\special{sh 0.200}%
\special{pa 2065 2274}%
\special{pa 2565 3271}%
\special{pa 1055 3271}%
\special{pa 1055 1762}%
\special{pa 2065 2274}%
\special{ip}%
\special{pn 8}%
\special{pa 2065 2274}%
\special{pa 2565 3271}%
\special{pa 1055 3271}%
\special{pa 1055 1762}%
\special{pa 2065 2274}%
\special{ip}%
%
\special{pn 0}%
\special{sh 0.200}%
\special{pa 2065 2215}%
\special{pa 2065 1181}%
\special{pa 1055 1181}%
\special{pa 1055 1704}%
\special{pa 2065 2215}%
\special{ip}%
\special{pn 8}%
\special{pa 2065 2215}%
\special{pa 2065 1181}%
\special{pa 1055 1181}%
\special{pa 1055 1704}%
\special{pa 2065 2215}%
\special{ip}%
%
\special{pn 0}%
\special{sh 0.200}%
\special{pa 2135 2274}%
\special{pa 3134 2274}%
\special{pa 3134 3271}%
\special{pa 2628 3271}%
\special{pa 2135 2274}%
\special{ip}%
\special{pn 8}%
\special{pa 2135 2274}%
\special{pa 3134 2274}%
\special{pa 3134 3271}%
\special{pa 2628 3271}%
\special{pa 2135 2274}%
\special{ip}%
%
\special{pn 8}%
\special{pa 2089 2250}%
\special{pa 1044 1728}%
\special{fp}%
%
\special{pn 8}%
\special{pa 2089 2250}%
\special{pa 2611 3295}%
\special{fp}%
%
\special{pn 8}%
\special{pa 2089 2250}%
\special{pa 2507 2250}%
\special{fp}%
\special{sh 1}%
\special{pa 2507 2250}%
\special{pa 2440 2230}%
\special{pa 2454 2250}%
\special{pa 2440 2270}%
\special{pa 2507 2250}%
\special{fp}%
%
\special{pn 8}%
\special{pa 2089 2250}%
\special{pa 2089 1832}%
\special{fp}%
\special{sh 1}%
\special{pa 2089 1832}%
\special{pa 2069 1899}%
\special{pa 2089 1885}%
\special{pa 2109 1899}%
\special{pa 2089 1832}%
\special{fp}%
%
\special{pn 8}%
\special{pa 2089 2250}%
\special{pa 1299 1856}%
\special{fp}%
\special{sh 1}%
\special{pa 1299 1856}%
\special{pa 1350 1904}%
\special{pa 1347 1880}%
\special{pa 1368 1868}%
\special{pa 1299 1856}%
\special{fp}%
%
\special{pn 8}%
\special{pa 2089 2250}%
\special{pa 2495 3063}%
\special{fp}%
\special{sh 1}%
\special{pa 2495 3063}%
\special{pa 2483 2994}%
\special{pa 2471 3015}%
\special{pa 2447 3012}%
\special{pa 2495 3063}%
\special{fp}%
\put(28.4400,-15.6400){\makebox(0,0)[lb]{$\ctc{\emptyset}{\emptyset}$}}%
\put(11.5000,-30.0000){\makebox(0,0)[lb]{$\ctc{\{1,2\}}{\{1,2\}}$}}%
\put(28.9000,-27.2000){\makebox(0,0)[lb]{$\ctc{\{2\}}{\{2\}}$}}%
\put(12.0600,-14.8400){\makebox(0,0)[lb]{$\ctc{\{1\}}{\{1\}}$}}%
\put(27.0300,-22.1500){\makebox(0,0)[lb]{$\ctc{\emptyset}{\{2\}}$}}%
\put(21.3500,-16.4600){\makebox(0,0)[lb]{$\ctc{\emptyset}{\{1\}}$}}%
\put(25.1700,-30.7500){\makebox(0,0)[lb]{$\ctc{\{2\}}{\{1,2\}}$}}%
\put(7.7000,-20.2000){\makebox(0,0)[lb]{$\ctc{\{1\}}{\{1,2\}}$}}%
%
\special{pn 8}%
\special{pa 2089 2250}%
\special{pa 2089 3295}%
\special{dt 0.045}%
%
\special{pn 8}%
\special{pa 2089 2250}%
\special{pa 1044 2250}%
\special{dt 0.045}%
%
\special{pn 8}%
\special{pa 2089 1170}%
\special{pa 2089 2250}%
\special{fp}%
%
\special{pn 8}%
\special{pa 3145 2250}%
\special{pa 2100 2250}%
\special{fp}%
\put(16.5500,-24.3100){\makebox(0,0)[lb]{$\ctc{\emptyset}{\{1,2\}}$}}%
%
\special{sh 1.000}%
\special{ia 2089 2250 34 34 0.0000000 6.2831853}%
\special{pn 8}%
\special{ar 2089 2250 34 34 0.0000000 6.2831853}%
\end{picture}}%
\caption{The fan $\fan$ for $n=3$}
\label{pic: fan2 and fan3}
\end{figure}
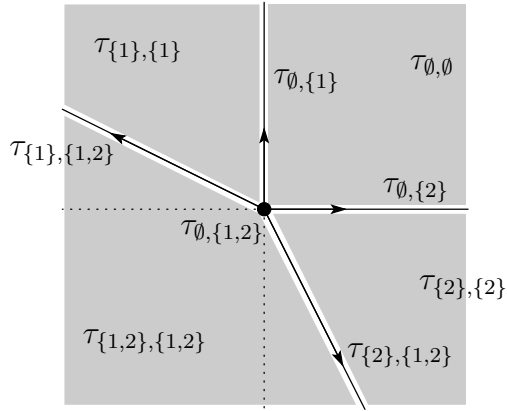
}
\end{example}

\vspace{5pt}

\subsection{Torus orbit decomposition of $X(\fan)$}\label{subsec: Orbit decomposition}
Recall from the last subsection that the toric orbifold $X(\fan)=(\C^{\ddyn}-\EL)/T$ is the quotient of the $T$-action on $\C^{\ddyn}-\EL$ given in \eqref{eq: def of T-action on C2r}.
For simplicity, we write elements of $\mathbb{C}^{\ddyn}$ by
\begin{align*}
 (x;y) = (x_1,\dots, x_{n-1}; y_1,\dots, y_{n-1}).
\end{align*}
Having the definition \eqref{eq: def of cone 2} of the cone $\ctc{K}{J}$ in mind, we consider the following subsets of $X(\fan)$.
For $K,J\subseteq \dyn$, we define a subset $X(\fan)_{K,J}\subseteq X(\fan)$ by \vspace{5pt}
\begin{equation}\label{eq: definition of X sigma IKJ}
X(\fan)_{K,J}\coloneqq\left\{[x;y]\in X(\fan) \ \left| \ \begin{cases}
x_i=0\quad \text{if}\quad i\in K\\
x_i\neq 0\quad \text{if}\quad i\notin K
\end{cases}
\hspace{-10pt},\hspace{5pt} 
\begin{cases}
y_i=0\quad \text{if}\quad i\notin J\\
y_i\neq 0\quad \text{if}\quad i\in J
\end{cases}\right\} \right. .\vspace{5pt}
\end{equation}
Namely, $K$ and $\dyn-J$ specify the coordinates having $0$ in their entries.
This subset is well-defined since the action of $T$ on $\C^{\ddyn}-\EL$ is given by a coordinate-wise multiplication of non-zero complex numbers (see \eqref{eq: def of T-action on C2r}).

For an element $[x;y]\in X(\fan)=(\C^{\ddyn}-\EL)/T$, it is obvious that we have $[x;y]\in X(\fan)_{K,J}$ for some $K,J\subseteq \dyn$. 
Thus, we obtain that
\begin{align*}
X(\fan) = \bigcup_{K,J\subseteq \dyn} X(\fan)_{K,J} .
\end{align*}

\begin{lemma}
For $K,J\subseteq \dyn$, we have $X(\fan)_{K,J}=\emptyset$ unless $K\subseteq J$. 
\end{lemma}

\begin{proof}
Suppose that $X(\fan)_{K,J}\ne\emptyset$, and we prove that $K\subseteq J$ holds.
This assumption means that there exists an element $[x;y]\in X(\fan)_{K,J}$. Since $(x;y)\in \C^{\ddyn}-\EL$, it satisfies the following property: if $x_i=0$ then $y_i\ne0$ (see \eqref{eq: def of C-E}). 
Recalling the definition~\eqref{eq: definition of X sigma IKJ} of $X(\fan)_{K,J}$, 
this means that $K\subseteq J$, as desired.
\end{proof}

\vspace{10pt}

Recall that $X(\fan)=(\C^{\ddyn}-\EL)/T$ has the canonical action of the torus $(\C^{\times})^{\ddyn}/T$ given by the coordinate-wise multiplication.
One can easily see that a non-empty subset of the form \eqref{eq: definition of X sigma IKJ} is an orbit of this torus action on $X(\fan)$.
More specifically, by
the quotient construction of simplicial toric varieties (see \cite[Section~5.1]{co-li-sc}), each $X(\fan)_{K,J}$ is the torus orbit in $X(\fan)$ corresponding to the cone $\ctc{K}{J}=\text{cone}(\{ -\alpha^{\vee}_i \mid i\in K \}\cup \{ \bm{e}_{i} \mid i \notin J \})$ in $\fan$ under the orbit-cone correspondence. In particular, we have
\begin{align*}
 \dim_{\C}X(\fan)_{K,J}=\codim_{\R} \ctc{K}{J} = |J|-|K|
\end{align*}
for $K\subseteq J\subseteq\dyn$.
In particular, we have $\codim_{\C}X(\fan)_{K,J}=|K|+(n-1)-|J|$ which agrees with the codimension seen from the definition \eqref{eq: definition of X sigma IKJ}.

\vspace{10pt}

\begin{corollary}\label{cor: orbit decomp of X}
The orbit decomposition of $X(\fan)=(\C^{\ddyn}-\EL)/T$ with respect to the canonical torus action of $(\C^{\times})^{\ddyn}/T$ is given by
\begin{align*}
X(\fan) = \bigsqcup_{K\subseteq J \subseteq \dyn} X(\fan)_{K,J} ,
\end{align*}
where $X(\fan)_{K,J}$ is the orbit corresponding to the cone $\ctc{K}{J}\in\fan$.
\end{corollary}

\vspace{10pt}

We end this section by recording the following claim.

\begin{lemma}\label{lem: preparation for nonnegative strata}
For $K\subseteq J\subseteq \dyn$, we have 
\begin{equation*}
X(\fan)_{K,J}=\left\{[x;y]\in X(\fan) \ \left| \ 
\begin{cases}
x_i=0\quad \text{if}\quad i\in K\\
x_i\ne 0\quad \text{if}\quad i\notin K
\end{cases}
\hspace{-10pt},\hspace{5pt} 
\begin{cases}
y_i=0\quad \text{if}\quad i\notin J\\
y_i=1\quad \text{if}\quad i\in J
\end{cases}\right\} \right. .
\end{equation*}
\end{lemma}

\begin{proof}
Let us denote the right hand side by $X'_{K,J}$, and we show that $X(\fan)_{K,J}=X'_{K,J}$.
By definition, we have $X'_{K,J}\subseteq X(\fan)_{K,J}$.
Since we have $[x;y]=[t\cdot(x;y)]$ for all $t\in T$, the opposite inclusion follows from the surjectivity of the homomorphism
\begin{align}\label{eq: alpha surj}
 T\rightarrow (\C^{\times})^{n-1}
 \quad ; \quad
 t\mapsto (\alpha_1(t),\ldots,\alpha_{n-1}(t)).
\end{align}
\end{proof}

\begin{example}\label{ex: n=3 orbits}
{\rm
Let $n=3$, then the orbits in $X(\fan)$ of dimension $0$ (i.e., the fixed points) are
\begin{equation*}
 X(\fan)_{\emptyset, \emptyset}, \ 
 X(\fan)_{\{1\}, \{1\}}, \ 
 X(\fan)_{\{2\}, \{2\}}, \ 
 X(\fan)_{\{1,2\}, \{1,2\}}.
\end{equation*}
The orbits in $X(\fan)$ of dimension $1$ are
\begin{equation*}
X(\fan)_{\emptyset, \{1\}}, \ 
X(\fan)_{\emptyset, \{2\}}, \ 
X(\fan)_{\{1\}, \{1,2\}}, \ 
X(\fan)_{\{2\}, \{1,2\}}.
\end{equation*}
Finally, the unique free orbit in $X(\fan)$ is $X(\fan)_{\emptyset, \{1,2\}}$. Compare these with the cones in Figure~\ref{pic: fan2 and fan3}.
}
\end{example}

\vspace{20pt}

\section{Nonnegative part of $X(\fan)$}
In this section, we study the nonnegative part of the toric orbifold $X(\fan)$. 
As is well-known, the torus orbit decomposition of a projective toric variety gives rise to a cell decomposition of the nonnegative part of $X(\fan)$. We describe this explicitly.

\subsection{Definition of $X(\fan)_{\ge 0}$}
For an arbitrary (normal) toric variety $X$ over $\C$, the nonnegative part $X_{\ge 0}$ is defined as the set of ``$\R_{\ge 0}$-valued points" of $X$. More details on $X_{\ge 0}$ can be found in \cite[Sect.~12.2]{co-li-sc}.

Recall that we have $X(\fan)=(\C^{\ddyn}-\EL)/T$.
By \cite[Proposition 12.2.1]{co-li-sc}, we have 
\begin{align*}
X(\fan)_{\ge0} = \{[x;y]\in X(\fan) \mid x_i, y_i\ge0\ \text{for $i\in \dyn$}\}.
\end{align*}
Namely, an element of $X(\fan)$ belongs to $X(\fan)_{\ge0}$ if and only if it can be represented by some $(x;y)\in\C^{\ddyn}-\EL$ with nonnegative entries.

\subsection{A decomposition of $X(\fan)_{\ge 0}$}
Recall from Corollary~\ref{cor: orbit decomp of X} that the torus orbit decomposition of $X(\fan)$ is given by 
\begin{align}\label{eq: X nonnegative decomp}
X(\fan) = \bigsqcup_{K\subseteq J \subseteq \dyn} X(\fan)_{K,J} .
\end{align}
To obtain a decomposition of $X(\fan)_{\ge0}$, we set 
\begin{align*}
\Xp{K,J} \coloneqq X(\fan)_{K,J}\cap X(\fan)_{\ge 0}
\quad \text{for $K\subseteq J \subseteq \dyn$.}
\end{align*}
Then,  we obtain from \eqref{eq: X nonnegative decomp} that
\begin{align}\label{eq: decomp of X nonnegative KJ}
X(\fan)_{\ge0} = \bigsqcup_{K\subseteq J \subseteq \dyn} \Xp{K,J} .
\end{align}

We define $T_{>0}\subseteq T (\subseteq \G{n})$ by
\begin{equation*}
 T_{>0} \coloneqq \{ \text{diag}(t_1,\ldots,t_n)\in \G{n} \mid \text{$t_i>0$ for $i\in \dyn$}\}
\end{equation*}
(which agrees with Lusztig's definition of $T_{>0}$ given in Section~\ref{sec: nonnegative part of Y}).
It is straightforward to see that the surjective homomorphism $T\rightarrow (\C^{\times})^{n-1}$ given in \eqref{eq: alpha surj} restricts to
\begin{align}\label{eq: varpi iso 3}
 T_{>0}\rightarrow (\R_{>0})^{n-1}
 \quad ; \quad
 t\mapsto (\alpha_1(t),\ldots,\alpha_{n-1}(t)).
\end{align} 

\vspace{5pt}

\begin{lemma}\label{eq: alpha surjectivity}
The map \eqref{eq: varpi iso 3} is a bijection.
\end{lemma}

\begin{proof}
To begin with, we consider the following map
\begin{align}\label{eq: varpi iso 5}
 T_{>0}\rightarrow (\R_{>0})^{n-1}
 \quad ; \quad
 t\mapsto (\varpi_1(t),\ldots,\varpi_{n-1}(t)).
\end{align}
Since we have $\varpi_i(t)=t_1t_2\cdots t_i$ for $i\in\dyn$, it is straightforward to see that this map is a bijection.
Note that each component of \eqref{eq: varpi iso 3} can be written as
\begin{align*}
 \alpha_i(t)=\varpi_{i-1}(t)^{-1}\varpi_i(t)^2\varpi_{i+1}(t)^{-1}
\end{align*}
with the convention $\varpi_0(t)=\varpi_n(t)=1$.
This leads us to consider the functions $\phi_i(u)\coloneqq u_{i-1}^{-1}u_i^2u_{i+1}^{-1}$ for $i\in\dyn$ and $u\in (\R_{>0})^{n-1}$ with the convention $u_0=u_n=1$, and then the homomorphism \eqref{eq: varpi iso 3} is the composition of \eqref{eq: varpi iso 5} and the map 
\begin{align}\label{eq: varpi iso 6}
 (\R_{>0})^{n-1} \rightarrow (\R_{>0})^{n-1}
 \quad ; \quad
 u \mapsto (\phi_1(u),\ldots,\phi_{n-1}(u)).
\end{align}
Since \eqref{eq: varpi iso 5} is bijective, it now suffices to show that the map \eqref{eq: varpi iso 6} is a bijection.
Under the identification $\R_{>0}\cong\R$ given by the logarithm, \eqref{eq: varpi iso 6} is identified with the linear map
\begin{align*}
 \R^{n-1} \rightarrow \R^{n-1}
 \quad ; \quad
 \bm{x} \mapsto C\bm{x},
\end{align*}
where $C$ is the Cartan matrix of type A$_{n-1}$.
Since we have $\det C\ne0$, this is a bijection. Hence, the claim follows.
\end{proof}

\vspace{10pt}

\begin{proposition}\label{prop: nonnengative stratum}
For $K\subseteq J\subseteq \dyn$, we have 
\begin{equation*}
\Xp{K,J}=\left\{[x;y]\in X(\fan) \ \left| \ \begin{cases}
x_i=0\quad \text{if}\quad i\in K\\
x_i>0\quad \text{if}\quad i\notin K
\end{cases}
\hspace{-10pt},\hspace{5pt} 
\begin{cases}
y_i=0\quad \text{if}\quad i\notin J\\
y_i=1\quad \text{if}\quad i\in J
\end{cases}\right\} \right. .
\end{equation*}
\end{proposition}

\begin{proof}
Denote the right hand of the desired equality by $\Xpp{K,J}$, and let us show that 
$\Xp{K,J} =\Xpp{K,J}$.
We have
\begin{equation*}
\Xpp{K,J} \subseteq X(\fan)_{K,J}\cap X(\fan)_{\ge 0}=\Xp{K,J}
\end{equation*}
by the definitions of $X(\fan)_{K,J}$ and $\Xp{K,J}$.
Hence, it suffices to show that $\Xp{K,J} \subseteq \Xpp{K,J}$.
To this end, take an arbitrary element 
\begin{equation}\label{eq: xy in intersection}
[x;y]\in \Xp{K,J}=X(\fan)_{K,J}\cap X(\fan)_{\ge 0}.
\end{equation}
Since $[x;y]\in X(\fan)_{K,J}$, Lemma~\ref{lem: preparation for nonnegative strata} implies that we may assume that $x$ and $y$ satisfy 
\begin{equation}\label{eq: nonnegative stratum 10}
\begin{cases}
x_i=0\quad \text{if}\quad i\in K\\
x_i\neq 0\quad \text{if}\quad i\notin K
\end{cases}
\quad \text{and} \quad
\begin{cases}
y_i=0\quad \text{if}\quad i\notin J\\
y_i=1\quad \text{if}\quad i\in J .
\end{cases}
\end{equation}
By \eqref{eq: xy in intersection}, we also have $[x;y]\in X(\fan)_{\ge 0}$ which means that there exists $(x';y')\in\C^{\ddyn}-\EL$ such that
\begin{itemize}\vspace{5pt}
\item[(1)] $x'_i,y'_i\ge0$ for $i\in \dyn$, \vspace{5pt}
\item[(2)] $[x;y] = [x';y']$.
\end{itemize}\vspace{5pt}
The condition (2) implies that we have $x_i\ne0$ if and only if $x'_i \ne0$ (similarly, $y_i\ne0$ if and only if $y'_i \ne0$) since the action of $T$ on $\C^{\ddyn}-\EL$ (given in \eqref{eq: def of T-action on C2r}) preserves the condition that a given component is zero or not.
Hence, these two conditions and \eqref{eq: nonnegative stratum 10} imply that we have
\begin{equation*}
\begin{cases}
x'_i=0\quad \text{if}\quad i\in K\\
x'_i>0\quad \text{if}\quad i\notin K
\end{cases}
\quad \text{and} \quad
\begin{cases}
y'_i=0\quad \text{if}\quad i\notin J\\
y'_i>0\quad \text{if}\quad i\in J .
\end{cases}
\end{equation*}
Comparing this and \eqref{eq: nonnegative stratum 10}, it follows from Lemma~\ref{eq: alpha surjectivity} that there exits $t\in T_{>0}$ such that
\begin{align*}
 \big( \alpha_1(t)y'_1,\ldots,\alpha_{n-1}(t)y'_{n-1} \big) = \big( y_1,\ldots,y_{n-1} \big).
\end{align*}
By construction, we have 
\begin{align*}
 [x';y']=[t\cdot (x';y')]=[x'',y], 
\end{align*}
where we set $x''\coloneqq \big( \varpi_1(t)x'_1,\ldots,\varpi_{n-1}(t)x'_{n-1} \big)$.
Since we have $t\in T_{>0}$ and $x'_i\ge0$ for all $i\in \dyn$, it follows that 
\begin{equation*}
\begin{cases}
x''_i=0\quad \text{if}\quad i\in K\\
x''_i>0\quad \text{if}\quad i\notin K .
\end{cases}
\end{equation*}
This means that we have $[x'';y]\in \Xpp{K,J}$ since $y$ satisfies the latter condition in \eqref{eq: nonnegative stratum 10}.
 Now we obtained that $[x;y]=[x';y']=[x'';y]\in \Xpp{K,J}$, as desired.
\end{proof}

\vspace{10pt}

By this proposition, the largest stratum $\Xp{\emptyset,\dyn}$ is given by
\begin{equation*}
\Xp{\emptyset,\dyn}
=
\{[x_1,\ldots,x_{n-1};1,\ldots,1]\in X(\fan) \mid
x_i>0\ (i\in \dyn)
\} .
\end{equation*}
For the later use, we prepare the following lemma.

\begin{lemma}\label{lem: closure}
We have
\begin{equation*}
\overline{\Xp{\emptyset,\dyn}}= X(\fan)_{\ge0},
\end{equation*}
where the closure is taken in $X(\fan)_{\ge0}$.
\end{lemma}

\begin{proof}
It suffices to show that the inclusion $\overline{\Xp{\emptyset,\dyn}}\supseteq X(\fan)_{\ge0}$ holds.
For an element $[x;y]\in X(\fan)_{\ge0}$, the decomposition \eqref{eq: decomp of X nonnegative KJ} of $X(\fan)_{\ge0}$ means that we have  $[x;y]\in \Xp{K,J}$ for some $K\subseteq J\subseteq \dyn$. Thus, we may assume by Proposition~\ref{prop: nonnengative stratum} that
\begin{equation*}
\begin{cases}
x_i=0\quad \text{if}\quad i\in K\\
x_i>0\quad \text{if}\quad i\notin K
\end{cases}
\hspace{-10pt},\hspace{5pt} 
\begin{cases}
y_i=0\quad \text{if}\quad i\notin J\\
y_i=1\quad \text{if}\quad i\in J .
\end{cases}
\end{equation*}
Using these $x$ and $y$, we consider an element $[x';y']\in X(\fan)$ given by
\begin{equation*}
 x'_i 
=
\begin{cases}
 \varepsilon \quad &(i\in K)\\
 x_i &(i\notin K)
\end{cases}
\quad \text{and} \quad 
y'_i 
=
\begin{cases}
 \varepsilon \quad &(i\notin J)\\
 y_i &(i\in J),
\end{cases}
\end{equation*}
where $\varepsilon>0$ is a positive real number.
By construction, we have $x'_i, y'_i>0$ for all $i\in[n-1]$ which means that the element $[x';y']$ belongs to both of $X(\fan)_{\emptyset,\dyn}$ (see \eqref{eq: definition of X sigma IKJ}) and $X(\fan)_{\ge0}$. That is, we have
\begin{equation*}
 [x';y']\in X(\fan)_{\emptyset,\dyn}\cap X(\fan)_{\ge0}=\Xp{\emptyset,\dyn}.
\end{equation*}
By taking the limit $\varepsilon\rightarrow 0$, we obtain that
\begin{equation*}
 \lim_{\varepsilon\rightarrow 0} \ [x';y'] = [x;y]
\end{equation*}
which means that the given element $[x;y]$ lies in the closure $\overline{\Xp{\emptyset,\dyn}}$.
\end{proof}

\vspace{5pt}

For a projective toric variety $X'$, it is well-known that the nonnegative part $X'_{\ge0}$ has a structure of a polytope. In fact, there is a lattice polytope $P'$ whose normal fan describe the toric variety $X'$, and the moment map associated to $P'$ provides a homeomorphism between $X'_{\ge0}$ and $P'$ (e.g.\ \cite[Sect.\ 12.2]{co-li-sc} or \cite[Sect.\ 4.2]{fu}). Under the moment map, the nonnegative part of each torus orbit corresponds to a relative interior of each face of $P'$.
In \cite[Proposition~4.4]{ab-ze23}, the authors showed that the toric orbifold $X(\fan)$ is projective. In particular, there is a simple lattice polytope which is homeomorphic to $X(\fan)_{\ge0}$ in the above sense.
In the next subsection, we describe this polytope explicitly.

\vspace{10pt}

\subsection{The Polytope $P_{n-1}$}\label{subsec: polytope}
In this subsection, we construct a full-dimensional lattice polytope in $\R^{n-1}$ whose normal fan is $\fan$.
This polytope was originally constructed in Horiguchi-Masuda-Shareshian-Song (\cite{ho-ma-sh-so21}) by cutting the permutohedron. To describe the polytope explicitly, we rather start from a concrete definition.

We say that a subset $J\subseteq[n-1]$ is \textit{connected} if it is of the form
\begin{align*}
 J=\{a,a+1,\ldots,b\}
\end{align*}
for some $a,b\in[n-1]$. 
For a connected subset 
$J=\{a,a+1,\ldots,b\}\subseteq[n-1]$, we set
\begin{align*}
 v_J \coloneqq \sum_{i=a}^b (i+1-a )(b+1-i) e_i .
\end{align*}
Noticing that $|J|=b+1-a$, this can be written in the coordinate expression as
\begin{align*}
 v_J = (0,\ldots,0,p, 2(p-1), 3(p-2),\ldots,(p-1)2, p,0,\ldots,0)
 \qquad (p= |J|), 
\end{align*} 
where the non-zero coordinates start from the $a$-th component and end at the $b$-th component.
For example, if $n=8$, then for $J=\{\ \ , 2,3,4,5, \ \ ,\ \ \}\subseteq[7]$ we have $$v_J=(0,4,6,6,4,0,0)$$ since $|J|=4$, and for $J=\{\ \ , \ \ ,3,4,5,6,\ \ \}\subseteq[7]$ we have $$v_J=(0,0,4,6,6,4,0)$$ by the same reason.
By definition, we have
\begin{align*}
 \bm{e}_i\circ v_J = (i+1-a )(b+1-i)
 \quad \text{for $i\in J=\{a,a+1,\ldots,b\}$}.
\end{align*}
We note that the same equality holds even when $i=a-1$ and $i=b+1$ since both sides are zero in those cases, where we take the convention $\bm{e}_0=\bm{e}_n=0$. Thus, we obtain that
\begin{align}\label{eq: polytope 20}
 \bm{e}_i\circ v_J = (i+1-a )(b+1-i)
 \quad \text{for $i\in\{a-1,a,\ldots,b,b+1\}$}.
\end{align}
For a general subset $J\subseteq[n-1]$, we take the decomposition $J=J_1\sqcup\cdots\sqcup J_m$ into the connected components, and set
\begin{align*}
 v_J \coloneqq v_{J_1}+\cdots+v_{J_m}\in \R^{n-1}
\end{align*}
with the convention $v_{\emptyset}\coloneqq\bm{0}$.
For example, if $n=12$ and $J=\{\ \ ,2,3,4,5,\ \ ,\ \ ,8,9,10, \ \ \}$, then
\begin{align*}
 v_J
 =
 (0,4,6,6,4,0,0,3,4,3,0)
 \in\R^{11} .
\end{align*}

\begin{definition}
We set $ P_{n-1}$ as the convex hull of the points $v_J\in \R^{n-1}$ for $J\subseteq[n-1]$$:$
\begin{align*}
 P_{n-1} \coloneqq \text{\rm Conv}(v_J \mid J\subseteq [n-1]) \subseteq \R^{n-1} .
\end{align*}
\end{definition}

\vspace{10pt}

Note that the zero vector $\bm{0}$ and the standard vectors $\bm{e}_i$ ($1\le i\le n-1$) belong to $P_{n-1}$ since we have $v_{\emptyset}=\bm{0}$ and $v_{\{i\}}=\bm{e}_i$ for $1\le i\le n-1$. This means that 
$P_{n-1}$ contains an $(n-1)$-dimensional simplex, and hence it follows that $P_{n-1}$ is a full-dimensional polytope in $\R^{n-1}$, that is,
\begin{align}\label{eq: polytope 30}
 \dim P_{n-1} = n-1.
\end{align}

\vspace{10pt}

\begin{example}\label{ex: polytope n=3,4}
{\rm
The polytopes $P_2$ and $P_3$ are depicted in Figure~\ref{pic: P2 and P3}. We encourage the reader to find the points $v_{\emptyset}$, $v_{\{2\}}$, $v_{\{1,3\}}$ in the picture of $P_3$.
\begin{figure}[htbp]
\hspace{-20pt}
{\unitlength 0.1in%
\begin{picture}(52.2100,19.3000)(7.3000,-25.1000)%
%
\special{pn 8}%
\special{pa 1600 2200}%
\special{pa 1600 800}%
\special{fp}%
\special{sh 1}%
\special{pa 1600 800}%
\special{pa 1580 867}%
\special{pa 1600 853}%
\special{pa 1620 867}%
\special{pa 1600 800}%
\special{fp}%
%
\special{pn 8}%
\special{pa 1400 2000}%
\special{pa 2800 2000}%
\special{fp}%
\special{sh 1}%
\special{pa 2800 2000}%
\special{pa 2733 1980}%
\special{pa 2747 2000}%
\special{pa 2733 2020}%
\special{pa 2800 2000}%
\special{fp}%
%
\special{pn 0}%
\special{sh 0.200}%
\special{pa 2000 2000}%
\special{pa 1600 2000}%
\special{pa 1600 1600}%
\special{pa 2400 1200}%
\special{pa 2000 2000}%
\special{ip}%
\special{pn 8}%
\special{pa 2000 2000}%
\special{pa 1600 2000}%
\special{pa 1600 1600}%
\special{pa 2400 1200}%
\special{pa 2000 2000}%
\special{pa 1600 2000}%
\special{fp}%
\put(20.0000,-22.2000){\makebox(0,0)[lb]{$v_{\{1\}}=(1,0)$}}%
\put(7.3000,-16.0000){\makebox(0,0)[lb]{$v_{\{2\}}=(0,1)$}}%
\put(22.6000,-11.3000){\makebox(0,0)[lb]{$v_{\{1,2\}}=(2,2)$}}%
\put(8.3000,-22.0000){\makebox(0,0)[lb]{$v_{\emptyset}=(0,0)$}}%
%
\special{sh 1.000}%
\special{ia 1600 2000 20 20 0.0000000 6.2831853}%
\special{pn 8}%
\special{ar 1600 2000 20 20 0.0000000 6.2831853}%
%
\special{sh 1.000}%
\special{ia 1600 1600 20 20 0.0000000 6.2831853}%
\special{pn 8}%
\special{ar 1600 1600 20 20 0.0000000 6.2831853}%
%
\special{sh 1.000}%
\special{ia 2000 2000 20 20 0.0000000 6.2831853}%
\special{pn 8}%
\special{ar 2000 2000 20 20 0.0000000 6.2831853}%
%
\special{sh 1.000}%
\special{ia 2400 1200 20 20 0.0000000 6.2831853}%
\special{pn 8}%
\special{ar 2400 1200 20 20 0.0000000 6.2831853}%
\put(17.9000,-17.8000){\makebox(0,0)[lb]{$P_2$}}%
%
\special{pn 8}%
\special{pa 4125 2437}%
\special{pa 4125 580}%
\special{fp}%
\special{sh 1}%
\special{pa 4125 580}%
\special{pa 4105 647}%
\special{pa 4125 633}%
\special{pa 4145 647}%
\special{pa 4125 580}%
\special{fp}%
%
\special{pn 8}%
\special{pa 3790 1948}%
\special{pa 5647 2504}%
\special{fp}%
\special{sh 1}%
\special{pa 5647 2504}%
\special{pa 5589 2466}%
\special{pa 5596 2489}%
\special{pa 5577 2504}%
\special{pa 5647 2504}%
\special{fp}%
%
\special{pn 8}%
\special{pa 3790 2170}%
\special{pa 5951 1350}%
\special{fp}%
\special{sh 1}%
\special{pa 5951 1350}%
\special{pa 5882 1355}%
\special{pa 5901 1369}%
\special{pa 5896 1392}%
\special{pa 5951 1350}%
\special{fp}%
%
\special{pn 8}%
\special{pa 4436 2144}%
\special{pa 5200 2009}%
\special{dt 0.045}%
%
\special{pn 0}%
\special{sh 0.200}%
\special{pa 4436 1787}%
\special{pa 4123 1695}%
\special{pa 4609 1214}%
\special{pa 5718 795}%
\special{pa 4436 1787}%
\special{ip}%
\special{pn 8}%
\special{pa 4436 1787}%
\special{pa 4123 1695}%
\special{pa 4609 1214}%
\special{pa 5718 795}%
\special{pa 4436 1787}%
\special{pa 4123 1695}%
\special{fp}%
%
\special{pn 0}%
\special{sh 0.200}%
\special{pa 4436 1787}%
\special{pa 4123 1695}%
\special{pa 4123 2046}%
\special{pa 4436 2139}%
\special{pa 4436 1787}%
\special{ip}%
\special{pn 8}%
\special{pa 4436 1787}%
\special{pa 4123 1695}%
\special{pa 4123 2046}%
\special{pa 4436 2139}%
\special{pa 4436 1787}%
\special{pa 4123 1695}%
\special{fp}%
%
\special{pn 0}%
\special{sh 0.200}%
\special{pa 4436 1782}%
\special{pa 4436 2144}%
\special{pa 5195 2009}%
\special{pa 5718 795}%
\special{pa 4436 1787}%
\special{pa 4436 1782}%
\special{ip}%
\special{pn 8}%
\special{pa 4436 1782}%
\special{pa 4436 2144}%
\special{pa 5195 2009}%
\special{pa 5718 795}%
\special{pa 4436 1787}%
\special{pa 4436 2144}%
\special{fp}%
%
\special{pn 8}%
\special{pa 4123 2040}%
\special{pa 5400 1557}%
\special{dt 0.045}%
%
\special{pn 8}%
\special{pa 4486 1915}%
\special{pa 5201 2015}%
\special{dt 0.045}%
%
\special{pn 8}%
\special{pa 4486 1905}%
\special{pa 4603 1221}%
\special{dt 0.045}%
\put(33.5000,-16.6000){\makebox(0,0)[lb]{$v_{\{3\}}$}}%
\put(34.2000,-18.6000){\makebox(0,0)[lb]{$=(0,0,1)$}}%
%
\special{sh 1.000}%
\special{ia 4120 1690 20 20 0.0000000 6.2831853}%
\special{pn 8}%
\special{ar 4120 1690 20 20 0.0000000 6.2831853}%
%
\special{sh 1.000}%
\special{ia 4120 2040 20 20 0.0000000 6.2831853}%
\special{pn 8}%
\special{ar 4120 2040 20 20 0.0000000 6.2831853}%
%
\special{sh 1.000}%
\special{ia 4430 2140 20 20 0.0000000 6.2831853}%
\special{pn 8}%
\special{ar 4430 2140 20 20 0.0000000 6.2831853}%
%
\special{sh 1.000}%
\special{ia 4430 1790 20 20 0.0000000 6.2831853}%
\special{pn 8}%
\special{ar 4430 1790 20 20 0.0000000 6.2831853}%
%
\special{sh 1.000}%
\special{ia 4600 1220 20 20 0.0000000 6.2831853}%
\special{pn 8}%
\special{ar 4600 1220 20 20 0.0000000 6.2831853}%
%
\special{sh 1.000}%
\special{ia 5190 2010 20 20 0.0000000 6.2831853}%
\special{pn 8}%
\special{ar 5190 2010 20 20 0.0000000 6.2831853}%
%
\special{sh 1.000}%
\special{ia 5710 800 20 20 0.0000000 6.2831853}%
\special{pn 8}%
\special{ar 5710 800 20 20 0.0000000 6.2831853}%
%
\special{pn 0}%
\special{sh 0.200}%
\special{pa 4760 1600}%
\special{pa 5010 1600}%
\special{pa 5010 1790}%
\special{pa 4760 1790}%
\special{pa 4760 1600}%
\special{ip}%
\special{pn 8}%
\special{pa 4760 1600}%
\special{pa 5010 1600}%
\special{pa 5010 1790}%
\special{pa 4760 1790}%
\special{pa 4760 1600}%
\special{ip}%
\put(48.0000,-17.6000){\makebox(0,0)[lb]{$P_3$}}%
\put(41.6000,-9.8000){\makebox(0,0)[lb]{$v_{\{2,3\}}=(0,2,2)$}}%
\put(54.6000,-7.4000){\makebox(0,0)[lb]{$v_{\{1,2,3\}}=(3,4,3)$}}%
\put(52.5000,-21.3000){\makebox(0,0)[lb]{$v_{\{1,2\}}=(2,2,0)$}}%
%
\special{pn 0}%
\special{sh 0}%
\special{pa 4560 2160}%
\special{pa 5020 2160}%
\special{pa 5020 2390}%
\special{pa 4560 2390}%
\special{pa 4560 2160}%
\special{ip}%
\special{pn 8}%
\special{pa 4560 2160}%
\special{pa 5020 2160}%
\special{pa 5020 2390}%
\special{pa 4560 2390}%
\special{pa 4560 2160}%
\special{ip}%
\put(42.6000,-23.3000){\makebox(0,0)[lb]{$v_{\{1\}}=(1,0,0)$}}%
\end{picture}}%
\caption{The polytopes $P_2$ and $P_3$}
\label{pic: P2 and P3}
\end{figure}
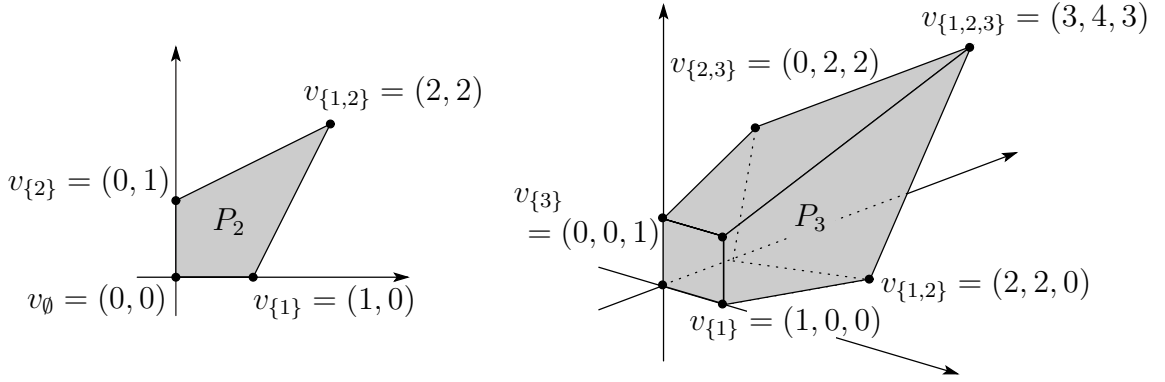
}
\end{example}

In the rest of this section, we set 
\begin{align*}
 -\alpha^{\vee}_i \coloneqq \bm{e}_{i-1}-2\bm{e}_i+\bm{e}_{i+1} \in \R^{\yn}
\end{align*}
for $i\in[n-1]$ with the convention $\bm{e}_0=\bm{e}_n=0$.
For $a,b\in\R^{n-1}$, we denote the standard inner product of $a$ and $b$ by $a\circ b$.

\begin{lemma}\label{lem: polytope ineq ei}
For $1\le i\le n-1$ and $J\subseteq[n-1]$, the following hold.
\begin{itemize}
 \item[$(1)$] We have $(-\alpha^{\vee}_i) \circ v_J \ge -2$. Moreover, the equality holds if and only if $i\in J$. \vspace{5pt}
 \item[$(2)$] We have $\bm{e}_i \circ v_J \ge0$. Moreover, the equality holds if and only if $i\notin J$.
\end{itemize}
\end{lemma}

\begin{proof}
Let us first prove the claim (2). Since $\bm{e}_i \circ v_J$ is the $i$-th component of $v_J$, the claim follows by the definition of $v_J$.

We prove the claim (1) in what follows.
We first consider the case $i\in J$, and we prove that $(-\alpha^{\vee}_i) \circ v_J=-2$.
Take the decomposition $J=J_1\sqcup\cdots \sqcup J_m$ into the connected components.
Since $i\in J$, we have $i\in J_{\ic}$ for some $1\le \ic\le m$.
Writing 
\begin{align*}
 J_{\ic}=\{a,a+1,\ldots,b\}
\end{align*}
for some $a,b\in \dyn$, we have from \eqref{eq: polytope 20} that
\begin{align*}
 (-\alpha^{\vee}_i) \circ v_{J_{\ic}}
 &= (\bm{e}_{i-1}-2\bm{e}_i+\bm{e}_{i+1})\circ v_{J_{\ic}} \\
 &= (i-a)(b+2-i) -2(i+1-a)(b+1-i) + (i+2-a)(b-i) \\
 &= -2
\end{align*}
by a direct calculation.
Since two connected components of $J$ are separated at least by an integer, the condition $i\in J_{\ic}$ implies that $i-1, i, i+1 \notin J_{\ell}$ for all $\ell\ne k$.
This and the claim (2) imply that
\begin{align*}
(-\alpha^{\vee}_i)\circ v_{J_{\ell}}
=
(\bm{e}_{i-1}-2\bm{e}_i+\bm{e}_{i+1})\circ v_{J_{\ell}} 
= 0
 \quad \text{for all $\ell\ne k$}.
\end{align*}
Since $v_J = v_{J_1}+\cdots+v_{J_m}$, we obtain $(-\alpha^{\vee}_i) \circ v_J=-2$ in this case.

For the case $i\notin J$, we prove that $(-\alpha^{\vee}_i) \circ v_J\ge0$ which completes the proof. This assumption means that we have $\bm{e}_i\circ v_J=0$ by the claim (2).
Since $-\alpha^{\vee}_i=\bm{e}_{i-1}-2\bm{e}_i+\bm{e}_{i+1}$, this implies that
\begin{align*}
 (-\alpha^{\vee}_i) \circ v_J = \bm{e}_{i-1}\circ v_J+0+\bm{e}_{i+1}\circ v_J \ge0,
\end{align*}
as desired.
\end{proof}

\vspace{10pt}

Motivated by Lemma~\ref{lem: polytope ineq ei}, we set 
\begin{align*}
 &\FP{i}
 \coloneqq 
 \text{\rm Conv}(v_J \mid i\in J) \subseteq P_{n-1}, \\
 &\FM{i}
 \coloneqq 
 \text{\rm Conv}(v_J \mid i\notin J) \subseteq P_{n-1}
\end{align*}
for $1\le i\le n-1$.
Lemma~\ref{lem: polytope ineq ei} means that these are faces of $P_{n-1}$.
For example, in the picture of $P_3$ depicted in Figure~\ref{pic: P2 and P3}, one can see that $\FP{3}$ and $\FM{3}$ are the top facet and the bottom facet of $P_{2}$ (with respect to the third axis), respectively.

\vspace{10pt}

\begin{proposition}\label{prop: polytope facets}
For $1\le i\le n-1$, $\FP{i}$ and $\FM{i}$ are facets of $P_{n-1}$.
\end{proposition}

\begin{proof}
Since Lemma~\ref{lem: polytope ineq ei} implies that $\FP{i}$ and $\FM{i}$ are non-empty proper faces of $P_{n-1}$, it suffices to show that their dimensions are equal to $n-2$.

To compute the dimension of $\FM{i}$, we observe that $\FM{i}$ contains $\bm{0}$ and $\bm{e}_k$ for all $k\ne i$ since $v_{\emptyset}=\bm{0}$ and $v_{\{k\}}=\bm{e}_k$. Thus, $\FM{i}$ contains an $(n-2)$-dimensional simplex. Since $\FM{i}$ is a proper face of $P_{n-1}$ as we saw above, we conclude that $\dim \FM{i}=n-2$.

To compute the dimension of $\FP{i}$, we consider a projection
$\R^{n-1}\rightarrow \R^{n-2}$ 
which maps $(x_1,\ldots,x_{n-1})\in \R^{n-1}$ to
\begin{align*}
 (x_1,\ldots,x_{i-1},x_{i+1},\ldots,x_{n-1})\in\R^{n-2}.
\end{align*}
In particular, this map sends $\bm{e}_i\in\R^{n-1}$ to the zero vector $\bm{0}\in\R^{n-2}$.
The image of $\FP{i}$ under this projection is a convex polytope since the map is linear.
It contains the zero vector $\bm{0}\in\R^{n-2}$ which is the image of $v_{\{i\}}=\bm{e}_i\in \FP{i}$. It also contains a nonzero multiple of each standard basis vector of $\R^{n-2}$ since it is the images of $v_{\{i,k\}}\in \FP{i}$ for some $k\ne i$ which is given by
\begin{align*}
 v_{\{i,k\}} = 
 \begin{cases}
  2\bm{e}_i+2\bm{e}_k \quad &\text{if $\{i,k\}$ is connected},\\
  \bm{e}_i+\bm{e}_k &\text{otherwise}.
 \end{cases}
\end{align*}
Hence, the image of $\FP{i}$ contains an $(n-2)$-dimensional simplex in $\R^{n-2}$. Since $\FP{i}$ is a proper face of $P_{n-1}$ as we saw above, this implies that $\dim \FP{i}=n-2$.
\end{proof}

\vspace{5pt}

Now, the following claim is a corollary of Lemma~\ref{lem: polytope ineq ei} and Proposition~\ref{prop: polytope facets}.

\begin{corollary}\label{cor: normal vector}
For $1\le i\le n-1$, the vectors $-\alpha^{\vee}_i$ and $\bm{e}_i$ are inward normal vectors of the facets $\FP{i}$ and $\FM{i}$, respectively.
\end{corollary}

\vspace{10pt}

We will show later that there is no other facets of $P_{n-1}$ other than $\FPM{i}$ for $1\le i\le n-1$ (see Proposition~\ref{prop: polytope ineq -alphai} below). Before proving it, let us describe a few properties of these facets $\FP{i}$ and $\FM{i}$.

\begin{lemma}\label{lem: up and bottom do not intersect}
For $i\in \dyn$, we have $\FP{i}\cap \FM{i}=\emptyset$.
\end{lemma}

\begin{proof}
We suppose that there is an element $x\in \FP{i}\cap \FM{i}$, and we deduce a contradiction.
Since $x\in \FP{i}=\text{\rm Conv}(v_J \mid i\in J)$, it can be written as an average of the vectors $v_J$ whose index $J$ contains $i$:
\begin{align*}
 x = \sum_{\substack{J\subseteq[n-1], \\ \text{$J$ contains $i$}}} \lambda_J v_J
\end{align*}
with $\lambda_J\ge0$ for all such $J$ and $\sum\lambda_J=1$.
For each $v_J$ appearing in this expression, we have $\bm{e}_i\circ v_J>0$ since $i\in J$ (Lemma~\ref{lem: polytope ineq ei}). 
Thus, it follows that
\begin{align*}
 \bm{e}_i\circ x > 0.
\end{align*}
We also have $x\in \FM{i}=\text{\rm Conv}(v_J \mid i\notin J)$ by the assumption, and an argument similar to that above shows 
\begin{align*}
 \bm{e}_i\circ x = 0
\end{align*}
since we have $\bm{e}_i\circ v_J=0$ for each $J$ which does not contain $i$ (Lemma~\ref{lem: polytope ineq ei}). 
Therefore, we obtain a contradiction, as desired.
\end{proof}

\vspace{10pt}

\begin{lemma}\label{lem: contain i J}
For $J\subseteq \dyn$ and $i\in \dyn$, we have
\begin{itemize}
 \item[(1)] $v_J \in \FP{i}$ if and only if $i\in J$, \vspace{5pt}
 \item[(2)] $v_J \in \FM{i}$ if and only if $i\notin J$.
\end{itemize}
\end{lemma}

\begin{proof}
We first prove (1). 
By definition, we have $v_J \in \FP{i}$ for all $i\in J$.
Moreover, if $v_J\in \FP{k}$ for some $k\in \dyn-J$, then it follows that $v_J\in \FP{k}\cap \FM{k}$ since $k\notin J$ which contradicts to $\FP{k}\cap \FM{k}=\emptyset$ (Lemma~\ref{lem: up and bottom do not intersect}).
The claim (2) follows by an argument similar to this.
\end{proof}

\vspace{10pt}

For subsets $K,L\subseteq\dyn$, 
Lemma~\ref{lem: up and bottom do not intersect} means that the 
intersection of $\FP{i}$ for $i\in K$ and $\FM{i}$ for $i\in L$ is empty unless $K\cap L=\emptyset$.
Writing $J=\dyn-L$, this condition is equivalent to $K\subseteq J$ which we also saw when we define the cone $\ctc{K}{J}$ in  \eqref{eq: def of cone 2}. This leads us to consider the following faces of $P_{n-1}$.
For $K\subseteq J\subseteq \dyn$, we set
\begin{align}\label{eq: def of FKJ'}
 \FCE{K}{J} \coloneqq \Bigg( \bigcap_{i\in K}\FP{i} \Bigg) \cap \Bigg( \bigcap_{i\notin J}\FM{i} \Bigg) \subseteq P_{n-1}
\end{align}
with the convention $\FCE{\emptyset}{\dyn}\coloneqq P_{n-1}$.
As we saw above, the condition $K\subseteq J$ ensures that there is no index $i$ appearing in common in \eqref{eq: def of FKJ'}.
Moreover, we note that $\FCE{K}{J}\ne\emptyset$ since we have $v_K,v_J\in\FCE{K}{J}$ because of the condition $K\subseteq J$. 

\vspace{10pt}

\begin{example}
{\rm
Let $n=4$. In this case, our polytope is $P_3$ depicted in Figure~\ref{pic: P2 and P3}. We visualize some faces of $P_{3}$.
For example, $\FCE{\{1\}}{\{1,2\}}=\FP{1}\cap \FM{3}$ is the edge joining two vertices $v_{\{1\}}$ and $v_{\{1,2\}}$.
Also, $\FCE{\{2,3\}}{\{2,3\}}=\FP{2}\cap\FP{3}\cap\FM{1}=\{v_{\{2,3\}}\}$ is a vertex.
The largest face $\FCE{\emptyset}{\{1,2\}}$ is the polytope $P_{2}$ itself  by definition.
}
\end{example}

\vspace{10pt}

We now prove the following by using the faces $\FCE{K}{J}$.

\begin{proposition}\label{prop: polytope ineq -alphai}
The facets of $P_{n-1}$ are exactly $\FPM{i}$ for $1\le i\le n-1$.
\end{proposition}

\begin{proof}
Let $\nfan$ be the normal fan of the polytope $P_{n-1}$, where we consider the set of integral points $\Z^{n-1}\subseteq \R^{n-1}$ as the lattice for the fan $\fan$. We show that the rays in $\nfan$ are in one-to-one correspondence to the facets $\FPM{i}$ for $1\le i\le n-1$.

For $K\subseteq J\subseteq \dyn$, we saw that the intersection $\FCE{K}{J}$ defined in \eqref{eq: def of FKJ'} is a non-empty face of $P_{n-1}$.
Let $\tau'_{K,J}\in\nfan$ be the cone corresponding to the face $\FCE{K}{J}$.
By \eqref{eq: def of FKJ'}, we have
\begin{align*}
 \FCE{K}{J} \subseteq \FP{i} \ \ \text{for $i\in K$}
 \quad \text{and} \quad
 \FCE{K}{J} \subseteq \FM{i} \ \ \text{for $i\notin J$}.
\end{align*}
In the normal fan $\fan'$, this means that
\begin{align*}
 \text{Cone}(-\alpha^{\vee}_i) \subseteq \tau'_{K,J} \ \ \text{for $i\in K$}
 \quad \text{and} \quad
 \text{Cone}(\bm{e}_{i}) \subseteq \tau'_{K,J} \ \ \text{for $i\notin J$}.
\end{align*}
In fact, these form a subset of rays of $\tau'_{K,J}$ (\cite[Proposition~2.3.7 (b)]{co-li-sc}), and we have
\begin{align}\label{eq: def of sigma KL 2}
 \ctc{K}{J} \subseteq \tau'_{K,J}, 
\end{align}
where $\ctc{K}{J} = \text{Cone}(\{ -\alpha^{\vee}_i \mid i\in K \}\cup \{ \bm{e}_{i} \mid i \notin J \})$ is the cone defined in \eqref{eq: def of cone 2}.
Recall from \eqref{eq: def of fan 1} and \eqref{eq: def of fan 2} that the fan $\fan$ is given by
\begin{align*}
 \fan = \{ \ctc{K}{J} \mid K\subseteq J\subseteq \dyn \}
 =\{ \sigma_{K,L} \mid K,J\subseteq \dyn, \ K\cap L=\emptyset \}.
\end{align*}
In \cite[Proposition~4.3, Remark~3.8]{ab-ze23}, the authors proved the fan $\fan$ is complete.
Let us denote by $\fan(k)$ and $\nfan(k)$ the set of $k$-dimensional cones in $\fan$ and $\nfan$, respectively ($0\le k\le n-1$).
The completeness of $\fan$ means that the set of full-dimensional cones $\ctc{K}{J}\in\fan(n-1)$ cover the entire space $\R^{n-1}$.
This and \eqref{eq: def of sigma KL 2} imply that
\begin{align*}
 \ctc{K}{J} = \tau'_{K,J}
 \quad \text{for all $\ctc{K}{J}\in \fan(n-1)$}.
\end{align*}
Since these full-dimensional cones $\ctc{K}{J}\in \fan(n-1)$ cover $\R^{n-1}$, there is no space left in $\R^{n-1}$ for which other full-dimensional cones in $\nfan(n-1)$ can fit in.
Thus, it follows that
\begin{align}\label{eq: def of sigma KL 5}
 \fan(n-1) = \nfan(n-1).
\end{align}

By definition, the rays in $\fan$ are $\text{cone}( -\alpha^{\vee}_i)$ and $\text{cone}( \bm{e}_i)$ for $1\le i\le n-1$ which are generated by normal vectors of the facets $\FP{i}$ and $\FM{i}$ of $P_{n-1}$, respectively. 
This means that we have $\fan(1)\subseteq \nfan(1)$.
To prove that there is no facet of $P_{n-1}$ other than $\FPM{i}$ for $1\le i\le n-1$, it is enough show that $\fan(1)=\nfan(1)$.
To see this, let $\rho\in\nfan(1)$ be an arbitrary ray. It corresponds to a facet of $P_{n-1}$, and this facet contains a vertex of $P_{n-1}$, say $v$. Since $P_{n-1}$ is a full-dimensional polytope in $\R^{n-1}$, the vertex $v$ corresponds to a full dimensional cone $\sigma_{\text{full}}\in\nfan$. This means that $\rho$ is a ray (i.e.\ a $1$-dimensional face) of $\sigma_{\text{full}}$ (\cite[Proposition~2.3.7 (b)]{co-li-sc}). 
Now, \eqref{eq: def of sigma KL 5} implies that we have $\sigma_{\text{full}}=\ctc{K}{J}$ for some $\ctc{K}{J}\in\fan(n-1)$. In particular, its ray $\rho$ belongs to $\fan(1)$. Hence, we conclude that $\fan(1)=\nfan(1)$, as desired.
\end{proof}

\vspace{5pt}

In the next claim, we take the set of integral points $\Z^{n-1}\subseteq \R^{n-1}$ as the lattice of $\R^{n-1}$.

\begin{corollary}\label{cor: simple polytope}
$P_{n-1}\subseteq\R^{n-1}$ is a 
full-dimensional simple lattice polytope.
\end{corollary}

\begin{proof}
By the definition of $P_{n-1}$ and \eqref{eq: polytope 30}, it is clear that $P_{n-1}$ is a lattice polytope of full-dimension in $\R^{n-1}$.
We show that the polytope $P_{n-1}$ is simple.
By definition, the set of vertices of $P_{n-1}$ is a subset of $\{v_J\mid J\subseteq[n-1]\}$.
It follows from Lemma~\ref{lem: contain i J} that
each $v_J\in P_{n-1}$ is contained in exactly $n-1$ facets $\FP{k}\ (k\in J)$ and $\FM{j}\ (j\notin J)$.
Thus, each vertex of $P_{n-1}$ is contained in exactly $n-1$ facets. Thus, the claim follows.
\end{proof}

\vspace{5pt}

Recall that a face of an arbitrary convex polytope must be an intersection of its facets (\cite[Theorem~2.7(v)]{zi95}). Proposition~\ref{prop: polytope ineq -alphai} now implies that the set of non-empty faces of $P_{n-1}$ are given by $\FCE{K}{J}$ defined in \eqref{eq: def of FKJ'} for $K\subseteq J\subseteq\dyn$.
We note that
\begin{align*}
\begin{split}
 \dim \FCE{K}{J} &= (n-1)-|K|-|\dyn-J| 
 = |J|-|K|
 \end{split}
\end{align*}
since $P_{n-1}$ is a simple polytope.
In particular, the vertices of $P_{n-1}$ (i.e. the faces of dimension 0) are given by $\FCE{J}{J}=\{v_J\}$ for $J\subseteq\dyn$ which means that the vertex set of $P_{n-1}$ is precisely $\{v_J\mid J\subseteq[n-1]\}$.
Since the set of non-empty faces of $P_{n-1}$ are given by $\FCE{K}{J}$, we obtain the following

\begin{corollary}\label{cor: normal fan}
The normal fan of $P_{n-1}$ coincides with the fan 
\begin{align*}
 \fan = \{ \ctc{K}{J} \mid K\subseteq J\subseteq \dyn  \}
\end{align*}
described in \eqref{eq: def of fan 2}, where $\tau_{K,J}=\text{\rm cone}(\{ -\alpha^{\vee}_i \mid i\in K \}\cup \{ \bm{e}_{i} \mid i \notin J \})$ is the cone corresponding to the face $\FCE{K}{J}\subseteq P_{n-1}$. 
\end{corollary}

\vspace{10pt}

We end this subsection by determining the combinatorial type of the polytope $P_{n-1}$. Let
\begin{equation*}
[0,1]^{n-1}
\coloneqq \{(x_1,\ldots,x_{n-1})\in \R^{n-1} \mid 0\le x_i\le 1 \ (i\in\dyn)\}
\end{equation*}
be the standard cube of dimension $n-1$.
There are $2(n-1)$ facets of $[0,1]^{n-1}$; the ones given by $x_i=0$ for $i\in\dyn$ and the ones given by $x_i=1$ for $i\in\dyn$.
Hence, the faces of $[0,1]^{n-1}$ are intersection of these facets, and each of them can be written as
\begin{equation}\label{eq: face post of cube}
E_{K,J}\coloneqq
\left\{(x_1,\ldots,x_{n-1})\in [0,1]^{\yn} \ \left| \ \begin{array}{l} x_i=0\quad\text{if $i\in K$},\\
x_i=1\quad\text{if $i\notin J$}
\end{array}
\right.
\right\}
\end{equation}
for some $K\subseteq J\subseteq \dyn$ (so that $E_{K,J}\ne\emptyset$).
It is straightforward to verify that these faces satisfy
\begin{equation}\label{eq: face post of cube 2}
E_{K,J}\subseteq E_{K',J'}
\quad \text{if and only if} \quad
K'\subseteq K\subseteq J\subseteq J',
\end{equation}
where the equality holds if and only if $K=K'$ and $J=J'$.
The next claim means that $P_{n-1}$ is combinatorially equivalent to the cube $[0,1]^{n-1}$.

\begin{proposition}\label{prop: polytope bijection}
For the set of faces of $P_{n-1}$, we have 
\begin{equation*}
\FCE{K}{J}\subseteq \FCE{K'}{J'}
\quad \text{if and only if} \quad
K'\subseteq K\subseteq J\subseteq J',
\end{equation*}
where the equality holds if and only if $K=K'$ and $J=J'$. 
\end{proposition}

\begin{proof}
If $K'\subseteq K\subseteq J\subseteq J'$, then it is clear that we have $\FCE{K}{J}\subseteq \FCE{K'}{J'}$ by the definition of $\FCE{K}{J}$ and $\FCE{K'}{J'}$.

Conversely, 
suppose that $\FCE{K}{J}\subseteq \FCE{K'}{J'}$.
Recall that we have 
\begin{align*}
v_K, v_J\in \FCE{K}{J}.
\end{align*}
In particular, it follows that $v_K\in \FCE{K'}{J'}$. Thus, the definition of $\FCE{K'}{J'}$ and Lemma~\ref{lem: contain i J} imply that $K'\subseteq K$.
Similarly, we have $v_J\in \FCE{K'}{J'}$, and hence we obtain that $J\subseteq J'$ by the definition of $\FCE{K'}{J'}$ and Lemma~\ref{lem: contain i J}. Thus, we conclude that $K'\subseteq K\subseteq J\subseteq J'$, as desired.

Lastly, we have $\FCE{K}{J}= \FCE{K'}{J'}$ if and only if $K=K'$ and $J=J'$ because of the first claim of this proposition.
\end{proof}

\vspace{10pt}

\subsection{Nonnegative part $X(\fan)_{\ge0}$ and the polytope $P_{n-1}$}\label{subsec: nonnegative and polytope}
In the last subsection, we showed that the normal fan of the polytope $P_{n-1}$ is precisely the fan $\fan$ for the toric orbifold $X(\fan)=(\C^{\ddyn}-\EL)/T$, where each $\ctc{K}{J}\in\fan$ is the cone corresponding to the face $\FCE{K}{J}\subseteq P_{n-1}$ for $K\subseteq J\subseteq \dyn$.
In Section~\ref{subsec: Orbit decomposition}, we also saw that $\ctc{K}{J}$ is the cone corresponding to the torus orbit $X(\fan)_{K,J}\subseteq X(\fan)$.
Therefore, $X(\fan)_{K,J}$ is the orbit corresponding to the face $\FCE{K}{J}$.
By using the language of relative interiors of faces, the correspondence becomes clearer; the (disjoint) decomposition of $X(\fan)$ by the torus orbits $X(\fan)_{K,J}$ corresponds to the (disjoint) decomposition of $P_{n-1}$ by the relative interiors $\Int(\FCE{K}{J})$ of the faces.

As is well-known in the theory of projective toric varieties (e.g.\ \cite[Sect.\ 12.2]{co-li-sc} or \cite[Sect.\ 4.2]{fu}), the above correspondence implies that there is the moment map  
\begin{equation*}
\mu \colon X(\fan)\rightarrow \R^{n-1}
\end{equation*}
determined by the full-dimensional simple lattice polytope $P_{n-1}\subseteq \R^{n-1}$ which satisfies the following property; it restricts to a homeomorphism 
\begin{equation*}
\overline{\mu} \colon X(\fan)_{\ge 0}\to P_{n-1}
\end{equation*}
such that 
\begin{equation*}
\overline{\mu}(\Xp{K,J})=\Int(\FCE{K}{J})\quad \text{for $K\subseteq J\subseteq \dyn$},
\end{equation*}
where $\Int(\FCE{K}{J})$ is the relative interior of the face $ \FCE{K}{J}\subseteq P_{n-1}$. 
In particular, each piece $\Xp{K,J}$ is homeomorphic to an open disk so that the decomposition \eqref{eq: decomp of X nonnegative KJ} of $X(\fan)_{\ge 0}$ is a cell decomposition.
See \cite[Sect.\ 12.2]{co-li-sc} or \cite[Sect.\ 4.2]{fu} for an explicit formula of the moment map $\mu$.

\vspace{10pt}

\begin{example}
{\rm
Let $n=4$ so that our polytope is $P_3$ depicted in Figure~\ref{pic: P2 and P3}.
The moment map gives us a homeomorphism
\begin{equation*}
 X(\fan)_{\ge0} \rightarrow P_3
\end{equation*}
which sends each stratum of $X(\fan)_{\ge0}$ to the corresponding face of $P_3$.
For example, by Proposition~\ref{prop: nonnengative stratum}, we have
\begin{equation*}
\Xp{\{1,2\},\{1,2,3\}}=
\left\{
[0,0,x_3;1,1,1]\in X(\fan) \mid
x_3>0
\right\} ,
\end{equation*}
and this is sent under the moment map to the open edge $\Int(\FCE{\{1,2\}}{\{1,2,3\}})=\Int(\FP{1}\cap\FP{2})$ connecting the vertices $v_{\{1,2\}}$ and $v_{\{1,2,3\}}$.
}
\end{example}

\vspace{20pt}

\section{Peterson Varieties}

In this section, we recall the definition of the Peterson variety, and we give a certain stratification of the Peterson variety according to K. Rietsch (\cite{ri03,ri06}).

\subsection{Peterson varietiy in $\G{n}/B^-$}
We keep the notation for $\G{n}=\G{n}(\C)$, $B^{\pm}$, $T$, $U^{\pm}$ from Section~\ref{sec: notations}.
Let $f$ be an $n\times n$ regular nilpotent matrix given by
\begin{align}\label{eq: def of f}
f = 
\begin{pmatrix}
0 &  & &  & \\
1 & 0 &  & \vspace{-5pt}\\ 
 & 1 & \ddots & & \vspace{-5pt}\\
 &    & \ddots & 0 & \\
 &    &  & 1 & 0 \\ 
\end{pmatrix}.
\end{align}

\begin{definition}\label{definition of complex Peterson variety}
The Peterson variety $Y$ in $\G{n}/B^-$ is defined as follows:
\begin{align*}
 Y\coloneqq  \{gB^-\in \G{n}/B^- \mid (\Ad{g^{-1}} f)_{i,j}=0 \ (j>i+1)\},
\end{align*}
where $(\Ad{g^{-1}} f)_{i,j}$ is the $(i,j)$-th component of the matrix $\Ad{g^{-1}} f=g^{-1}fg$.
\end{definition}

\vspace{10 pt}

\subsection{Richardson strata of $Y$}\label{subsec: A decomposition of XJ}\label{subsec: decomposition of Y}
We begin with fixing our notations.
For $i,j\in [n]$, we denote by $E_{i,j}$ the $n\times n$ matrix having $1$ at the $(i,j)$-th position and $0$'s elsewhere. Let 
$e_i \coloneqq E_{i,i+1}$ and $f_i \coloneqq E_{i+1,i}$
 for $i\in \dyn$.
We set
\begin{align*}
x_i(t)\coloneqq \exp(te_i) \in U^+(\subseteq \G{n})
\quad \text{and} \quad
y_i(t)\coloneqq \exp(tf_i) \in U^-(\subseteq \G{n})
\end{align*}
for $1\le i\le n-1$ and $t\in \C$.
For example, when $n=3$, we have 
\begin{align*}
x_1(t)= \begin{pmatrix}
      1 & t &0 \\
      0 & 1 &0\\
      0&0&1
   \end{pmatrix}
\quad \text{and} \quad
y_2(t)=\begin{pmatrix}
      1 & 0 &0 \\
      0 & 1 &0\\
      0 & t &1
   \end{pmatrix}
   \quad \text{for $t\in\C$}.
\end{align*}
For $i\in \dyn$, let 
\begin{align*}
\dot{s}_i\coloneqq  y_i(-1)x_i(1)y_i(-1).
\end{align*}
Then we have $\dot{s}_i\in N(T)$, where $N(T)$ is the normalizer of the maximal torus $T\subseteq \G{n}$.
For example, when $n=3$, then we have
\begin{align*}
\dot{s}_1=\begin{pmatrix}
0 & 1 & 0\\
-1& 0 & 0\\
0 & 0 & 1
\end{pmatrix}
\quad \text{and} \quad
\dot{s}_2=\begin{pmatrix}
1 & 0 & 0\\
0 & 0 & 1\\
0 & -1& 0
\end{pmatrix}.
\end{align*}
The Weyl group is defined by $W\coloneqq N(T)/T=\langle s_1,s_2,\cdots, s_{\yn}\rangle$, where $s_i \coloneqq \dot{s_i}T\in W$ for $i\in \dyn$.
In the rest of this paper, we identify the Weyl group $W$ and the permutation group $\Sn$ which corresponds $s_i$ to the simple reflection $(i,i+1)$ for $i\in\dyn$.

For $w\in \Sn$, we define a representative $\dot{w}\in \G{n}$ by
\begin{equation}\label{eq: def of representative}
 \dot{w}\coloneqq\dot{s}_{i_1}\dot{s}_{i_2}\cdots \dot{s}_{i_m}\in \G{n},
\end{equation}
where $s_{i_1}s_{i_2}\cdots s_{i_m}$ is a reduced expression of $w$. It follows from Matsumoto's criterion (\cite[Theorem~1.8]{ma99}) that the definition of $\dot{w}$ does not depend on a choice of a reduced expression of $w$ since the matrices $\dot{s}_1,\ldots,\dot{s}_{n-1}\in \G{n}$ satisfy 
\begin{itemize} \vspace{5pt}
 \item[(1)] $\dot{s}_i\dot{s}_j=\dot{s}_j\dot{s}_i$ for $i,j\in \dyn$ with $|i-j|\ge2$, \vspace{5pt}
 \item[(2)] $\dot{s}_i\dot{s}_{i+1}\dot{s}_i=\dot{s}_{i+1}\dot{s}_i\dot{s}_{i+1}$ for $i\in [n-2]$.
\end{itemize} \vspace{5pt}
For the longest element $w_0\in\mathfrak{S}_n$, the representative $\dot{w}_0$ is given by
\begin{equation}\label{eq: representative of w0}
\dot{w}_0=
\begin{pmatrix} 
  & & & & \!\!1\\
  & & & \!\!-1 & \\
  &  & \!\!1 & & \\
  & \!\!\!\!\rotatebox{75}{$\ddots$} & & & \vspace{-3pt}\\
  (-1)^{n-1}\!\!\!\! & & & & & 
\end{pmatrix}
\in \G{n}.
\end{equation}
We call $\dot{w}_0$ \textbf{the longest permutation matrix in $\G{n}$}.

\vspace{10pt}

For a subset $J\subseteq \dyn$, we may decompose it into the connected components as in Section~\ref{subsec: polytope}:
\begin{align}\label{eq: connected components}
 J = J_1 \sqcup J_2 \sqcup \cdots \sqcup J_{m},
\end{align}
where each $J_{\ic}\ (1\le \ic\le m)$ is a maximal connected subset of $J$. To determine each $J_{\ic}$ uniquely, we require that elements of $J_{\ic}$ are less than elements of $J_{{\ic}'}$ when $\ic< {\ic}'$.

For $J\subseteq \dyn$, let us introduce a Young subgroup of $\Sn$ given by
\begin{align*}
 \mathfrak{S}_J \coloneqq \mathfrak{S}_{J_1}\times \mathfrak{S}_{J_2}\times \cdots\times\mathfrak{S}_{J_m}
 \subseteq \Sn ,
\end{align*}
where each $\mathfrak{S}_{J_{\ic}}$ is the subgroup of $\Sn$ generated by the simple reflections $s_i$ for all $i\in J_{\ic}$. Let $w_J$ be the longest element of $\mathfrak{S}_J$, i.e., \ 
\begin{align*}
 w_J \coloneqq w_{J_1} w_{J_2} \cdots w_{J_{m}} \in \mathfrak{S}_J , 
\end{align*}
where each $w_{J_{\ic}}$ is the longest element of the permutation group $\mathfrak{S}_{J_{\ic}}$ $(1\le \ic\le m)$. 
It is well known that $w_K\le w_J$ in the Bruhat order if and only if $K\subseteq J$ (\cite[Lemma~2.4]{AHKZ} or \cite[Lemma~6]{Insko-Tymoczko}).
Listing reduced expressions of $w_{J_{\ic}}$ from $\ic=1$ to $\ic=m$, we obtain a reduced expression of $w_J$. This gives us a representative $\dot{w}_J\in \G{n}$ by the construction given in \eqref{eq: def of representative}.
We illustrate the form of $\dot{w}_J\in \G{n}$ by an example.

\begin{example}\label{eg: wJ}
\rm
Let $n=10$ and $J=\{1,2, \ \ ,4,5,6, \ \ , \ \ ,9\}$. Then we have 
\begin{align*}
 J = \{1,2\} \sqcup\{4,5,6\} \sqcup\{9\} = J_1\sqcup J_2\sqcup J_3.
\end{align*} 
Thus, we have 
\begin{align*}
{\tiny
 \dot{w}_J 
 = \dot{w}_{J_1} \dot{w}_{J_2} \dot{w}_{J_3} 
 =
 \left(
 \begin{array}{@{\,}ccc|cccc|c|cc@{\,}}
     & & 1\!\! & & & & & & & \\
     & \!\!\!\!-1\!\!\!\! & & & & & & & & \\ 
    1 & & & & & & & & & \\ \hline
     & & & & & & 1\!\!& & &  \\
     & & & & & \!\!\!\!-1\!\!\!\! & & & & \\
     & & & & 1 & & & & & \\
     & & & \!\!\!-1\!\!\!\! & & & & & & \\ \hline
     & & & & & & & \!\!1\!\! & & \\ \hline
     & & & & & & & & & 1 \\
     & & & & & & & & \!\!\!-1\!\!\!\! & 
 \end{array}
 \right)\in \G{10}.}
\end{align*}
\end{example}

\vspace{10pt}

For $J\subseteq \dyn$, let 
\begin{equation*}
X_{w_J}^{\circ} = B^-\dot{w}_JB^-/B^-
\quad \text{and} \quad 
\Omega_{w_J}^{\circ} = B^+\dot{w}_JB^-/B^-
\end{equation*}
be the Schubert cell and the opposite Schubert cell associated with the permutation $w_J$, respectively.
According to Rietsch (\cite{ri03,ri06}), 
we consider a \textbf{Richardson stratum}
\begin{equation}\label{eq: def of richardson strata}
\RS_{K,J}\coloneqq Y\cap \Omega_{w_K}^{\circ} \cap X_{w_J}^{\circ} \subseteq Y
\end{equation}
for $J, K\subseteq \dyn$.
We note that $\RS_{K,J}=\emptyset$ unless $w_K\le w_J$ (i.e., $K\subseteq J$) by \cite[Lemma~3.1(3)]{Richardson92}.
It is well-known that 
\begin{equation*}
Y=\bigsqcup_{K\subseteq \dyn}Y\cap \Omega_{w_K}^{\circ}
=\bigsqcup_{J\subseteq \dyn}Y\cap X_{w_J}^{\circ} 
\end{equation*}
(cf.\ \cite[Lemma~3.5]{AHKZ}).
Therefore, we obtain  
\begin{equation}\label{decomposition of Y as union of ZKJ}
Y=\bigsqcup_{K\subseteq J\subseteq \dyn}\RS_{K,J}.
\end{equation}
As a special case, we call $\RS_{K,\dyn}$ ($K\subseteq \dyn$) a \textbf{Richardson stratum  in the big cell}.

\vspace{10pt}

\subsection{$J$-Toeplitz matrices and $Y\cap X_{w_J}^{\circ}$}\label{subsec: on elements in Y ge 0 cap XJ}
For a positive integer $k$, a $k\times k$ matrix of the form
\begin{equation*}
\begin{pmatrix} 
  1&  &&&\\
  x_1 & \!\!\!1 &  && \\
  x_2 & \!\!\!x_1 & \!\!\!1 &  &\\
  \vdots & \!\!\!x_2 & \!\!\!x_1 & \ddots\\
  x_{k-2} & \!\!\!\vdots & \!\!\!\ddots & \ddots & \ddots\\
  x_{k-1} & \!\!\!x_{k-2}& \!\!\!\cdots & x_2 & x_1& 1
\end{pmatrix}
\qquad (x_1,\ldots,x_{k-1}\in\C)
\end{equation*}
is called a \textbf{Toeplitz matrix}.
We call the product $x\dot{w}_0$ for a $k\times k$ Toeplitz matrix $x$ and $\dot{w}_0\in \G{k}$ (see \eqref{eq: representative of w0}) an \textbf{anti-Toeplitz matrix}.
For example, 
\begin{equation*}
\begin{pmatrix}
  & & 1 \\
  & -1 & x_1 \\
  1 & -x_1 & x_2
\end{pmatrix}
\quad \text{and} \quad
\begin{pmatrix}
  & & & 1 \\
  & & -1 & x_1 \\
  & 1 & -x_1 & x_2 \\
  -1 & x_1 & -x_2 & x_3
\end{pmatrix}
\end{equation*}
are anti-Toeplitz matrices.

For each $J\subseteq[n-1]$, we construct a partition of $[n]$ as follows.
Take the decomposition $J = J_1 \sqcup J_2 \sqcup \cdots \sqcup J_{m}$ into the connected components (see \eqref{eq: connected components}), and set
\begin{equation}\label{def: definition of J'i}
\Jb{\ic}\coloneqq J_{\ic}\cup\left\{\max(J_{\ic})+1\right\}
\qquad (1\le \ic\le m).
\end{equation}
Let $\Jb{}\coloneqq \Jb{1}\sqcup \Jb{2}\sqcup \cdots \sqcup \Jb{m}\subseteq[n]$.
We write $[n]-\Jb{} =\{p_1,\ldots,p_s\}$.
Then we obtain an (unorderd) partition of $[n]$:
\begin{equation}\label{eq: partition from J}
 [n] = \Jb{1}\sqcup \cdots \sqcup \Jb{m} \sqcup \{p_1\} \sqcup \cdots \sqcup\{p_s\}.
\end{equation}
This allows us to consider block diagonal matrices with respect to the partition \eqref{eq: partition from J} of $[n]$.
We call such matrices \textbf{$J$-block matrices}.
In the same manner, \textbf{$J$-block diagonal matrices} and \textbf{$J$-block lower triangular matrices} are defined.

\begin{example}\label{ex: J block and lower}
\rm
Let $n=10$ and $J=\{1,2\} \sqcup\{4,5,6\} \sqcup\{9\}$ as above. 
Then we have $\Jb{}=\Jb{1}\sqcup \Jb{2}\sqcup \Jb{3}$ with 
\begin{equation*}
\Jb{1}=\{1,2,3\}, \ \Jb{2}=\{4,5,6,7\}, \ \Jb{3}=\{9,10\}.
\end{equation*}
Hence, we have $[10]-\Jb{}=\{8\}$.
Thus, $J$-block diagonal matrices and $J$-block lower triangular matrices are of the form
\begin{align*}
{\tiny
 \left(
 \begin{array}{@{\,}ccc|cccc|c|cc@{\,}}
     *\!\! & \!\!*\!\! & \!\!* & & & & & & & \\
     *\!\! & \!\!*\!\! & \!\!* & & & & & & & \\ 
     *\!\! & \!\!*\!\! & \!\!* & & & & & & & \\ \hline
     & & & *\!\! & \!\!*\!\! & \!\!*\!\! & \!\!* & & &  \\
     & & & *\!\! & \!\!*\!\! & \!\!*\!\! & \!\!* & & & \\
     & & & *\!\! & \!\!*\!\! & \!\!*\!\! & \!\!* & & & \\
     & & & *\!\! & \!\!*\!\! & \!\!*\!\! & \!\!* & & & \\ \hline
     & & & & & & & \!\!*\!\! & & \\ \hline
     & & & & & & & & *\!\! & \!\!* \\
     & & & & & & & & *\!\! & \!\!*
 \end{array}
 \right)}
 \quad \text{and} \quad 
 {\tiny
 \left(
 \begin{array}{@{\,}ccc|cccc|c|cc@{\,}}
     *\! & \!\!*\!\! & \!* & & & & & & & \\
     *\! & \!\!*\!\! & \!* & & & & & & & \\ 
     *\! & \!\!*\!\! & \!* & & & & & & & \\ \hline
     *\! & \!\!*\!\! & \!* & *\! & \!\!*\!\! & \!*\! & \!\!* & & &  \\
     *\! & \!\!*\!\! & \!* & *\! & \!\!*\!\! & \!*\! & \!\!* & & & \\
     *\! & \!\!*\!\! & \!* & *\! & \!\!*\!\! & \!*\! & \!\!* & & & \\
     *\! & \!\!*\!\! & \!* & *\! & \!\!*\!\! & \!*\! & \!\!* & & & \\ \hline
     *\! & \!\!*\!\! & \!* & *\! & \!\!*\!\! & \!*\! & \!\!* & \!\!*\!\! & & \\ \hline
     *\! & \!\!*\!\! & \!* & *\! & \!\!*\!\! & \!*\! & \!\!* & \!\!*\!\! & *\!\! & \!\!* \\
     *\! & \!\!*\!\! & \!* & *\! & \!\!*\!\! & \!*\! & \!\!* & \!\!*\!\! & *\!\! & \!\!*
 \end{array}
 \right),
 }
\end{align*}
where all $*$'s can take arbitrary complex numbers.
\end{example}

\vspace{10pt}

\begin{definition}
{\rm 
For $J\subseteq[n-1]$, a \textbf{$J$-Toeplitz matrix} is a $J$-block diagonal matrix whose diagonal blocks consist of Toeplitz matrices.
}
\end{definition}

\vspace{10pt}

\begin{example}\label{ex: J-Toeplitz matrices}
\rm
Let $n=10$ and $J=\{1,2\}\sqcup\{4,5,6\}\sqcup\{9\}$ as above. 
Then a $J$-Toeplitz matrix is of the form
\begin{align*}
 {\tiny
 \left(
 \begin{array}{@{\,}ccc|cccc|c|cc@{\,}}
     1& &  & & & & & & & \\ 
     a & \!\!1\!\!& & & & & & & & \\ 
     b & \!\!a\!\! & \!1& & & & & & & \\ \hline
     & & & 1 & & & & & &  \\
     & & & x & \!\!1\!\! &  & & & & \\
     & & & y & \!\!x\!\! & 1& & & & \\
     & & & z & \!\!y\!\! & x & \!\!\!1 & & & \\ \hline
     & & & & & & & \!\!1\!\! & & \\ \hline
     & & & & & & & & 1\!\! &  \\
     & & & & & & & & s\!\! & \!1
 \end{array}
 \right)
 }
 \qquad (a,b,x,y,z,s\in \C).
\end{align*}
\end{example}

\vspace{15pt}

Since $w_J$ is the longest element of $\mathfrak{S}_{J}(\subseteq \mathfrak{S}_{n})$, the associated Schubert cell $X_{w_J}^{\circ}=B^-\dot{w}_JB^-/B^-$ is the set of $gB^- \in \G{n}/B^-$ such that $g$ is $J$-block diagonal:
\begin{equation*}
X_{w_J}^{\circ} = \{ gB^- \in \G{n}/B^- \mid \text{$g$ is $J$-block diagonal}\}.
\end{equation*}
By a direct computation, one can verify the following well-known fact; 
an element $gB^-\in \G{n}/B^-$ belongs to $Y\cap X_{w_J}^{\circ}$ if and only if it can be represented by a matrix $g\in \G{n}$ such that $g$ is a $J$-block diagonal matrix whose diagonal blocks consist of anti-Toeplitz matrices.
We illustrate this in the next example.

\vspace{10pt}

\begin{example}\label{eq: elements of XJ}
\rm
Let $n=10$ and $J=\{1,2\}\sqcup\{4,5,6\}\sqcup\{9\}$ as above. Then $Y\cap X_{w_J}^{\circ}$ is the set of matrices given by
\begin{align*}
{\tiny
 \left(
 \begin{array}{@{\,}ccc|cccc|c|cc@{\,}}
     & & \!\!1 & & & & & & & \\
     & \!\!\!\!-1 & \!\!a & & & & & & & \\ 
    1 & \!\!\!\!-a & \!\!b & & & & & & & \\ \hline
     & & & & & & \!\!\!1& & &  \\
     & & & & & \!\!\!\!-1 & \!\!\!x & & & \\
     & & & & \!\!\!\!1 & \!\!\!\!-x & \!\!\!y & & & \\
     & & & \!\!-1 & \!\!\!\!x & \!\!\!\!-y & \!\!\!z & & & \\ \hline
     & & & & & & & \!\!1\!\! & & \\ \hline
     & & & & & & & & & \!\!\!\!\!1 \\
     & & & & & & & & \!-1 & \!\!\!\!\!s 
 \end{array}
 \right)B^-
 }
 \qquad (a,b,x,y,z,s\in \C).
\end{align*}
\end{example}

\vspace{15pt}

In other words, we can write
\begin{equation}\label{mathcal{X}J=J-toeplitz}
Y\cap X_{w_J}^{\circ}
=\{x\dot{w}_JB^- \in \G{n}/B^- \mid \text{$x$ is $J$-Toeplitz}\}. 
\end{equation}
For example, the matrix exhibited in the previous example is
the product of the $J$-Toeplitz matrix in Example~\ref{ex: J-Toeplitz matrices} and the permutation matrix $\dot{w}_J$ in Example~\ref{eg: wJ}.

\vspace{15pt}

\subsection{Description of the strata $\RS_{K,J}$}\label{subsec: description of RKJ}
For $i\in\dyn$, let
\begin{equation}\label{eq:4.3 first}
\Delta_{i} \colon \G{n}\to\C
\end{equation}
be the function which takes the lower right $(n-i)\times (n-i)$ minor of a given matrix in $\G{n}$.
For example, for $g=(a_{i,j})_{1\le i\le 4,1\le j\le 4}\in \G{4}$,
we have
\begin{equation}\label{eq:4.3 second}
\Delta_1(g) = \det
\begin{pmatrix}
 a_{22} & a_{23} & a_{24} \\
 a_{32} & a_{33} & a_{34} \\
 a_{42} & a_{43} & a_{44}
\end{pmatrix}, \ 
\Delta_2(g) = \det
\begin{pmatrix}
 a_{33} & a_{34} \\
 a_{43} & a_{44}
\end{pmatrix}, \ 
\Delta_3(g) = a_{44}.
\end{equation}

To study the strata $\RS_{K,J}=Y\cap \Omega_{w_K}^{\circ} \cap X_{w_J}^{\circ} $, we first recall Rietsch's description of Richardson strata in the big cell from \cite{ri03}.
The claim \cite[Lemma~4.5]{ri03} combined with \cite[Theorem 4.6(1)]{ri03} due to Dale Peterson gives the following description of Richardson strata in the big cell.
\begin{proposition}\label{lemm: Rietsch's result for Z KI}
Let $K\subseteq \dyn$. Then, we have
\begin{equation*}
\RS_{K,\dyn}
=\left\{x\dot{w}_0 B^- \ \left| \ \text{$x$ is Toeplitz and}~\begin{cases}\Delta_{i}(x\dot{w}_0)=0\quad\text{if}\quad i\in K\\
\Delta_{i}(x\dot{w}_0)\ne 0\quad\text{if}\quad i\notin K
\end{cases}
\right.
\right\}.
\end{equation*}
\end{proposition}

\vspace{10pt}

\begin{example}\label{eg: strata in big cell}
\rm
Let $n=3$. Then we have 
\begin{equation*}
\RS_{\{1\},\{1,2\}}
=\left\{
  \left.
  \begin{pmatrix}
      0&0&1&\\
      0&-1&a\\
       1&-a&b
   \end{pmatrix}B^- \ \right| \ 
   \Delta_1=a^2-b= 0, \ 
   \Delta_2=b\ne0
   \right\}
\end{equation*}  
and 
\begin{equation*}
\RS_{\emptyset,\{1,2\}}
=\left\{
  \left.
  \begin{pmatrix}
      0&0&1&\\
      0&-1&a\\
       1&-a&b
  \end{pmatrix}B^- \ \right| \ 
  \Delta_1=a^2-b\ne 0, \ \Delta_2=b\ne0 
  \right\}
\end{equation*}
(cf.\ Example~\ref{ex: n=3 orbits} in Section~\ref{subsec: Orbit decomposition}).
\end{example}

\vspace{10pt}

We apply Rietsch's result (Proposition~\ref{lemm: Rietsch's result for Z KI}) to characterize $\RS_{K,J}$ for all $K\subseteq J\subseteq \dyn$ in the next claim.
For that purpose, let us prepare some notations.
Let $J\subseteq\dyn$, and take the decomposition $$J = J_1 \sqcup J_2 \sqcup \cdots \sqcup J_{m}$$ into the connected components. 
For a matrix $g\in \G{n}$, we denote the $\Jb{\ic}\times \Jb{\ic}\ (\subseteq[n]\times [n])$ diagonal block of $g$ by $g(\Jb{\ic})$ for $1\le \ic\le m$ (see \eqref{def: definition of J'i} for the definition of $\Jb{\ic}$).
For each connected component $J_{\ic}\subseteq J$, we set 
\begin{equation*}
n_{\ic}\coloneqq |\Jb{\ic}| .
\end{equation*}  
We consider the special linear group $\G{n_{\ic}}$ of degree $n_{\ic}$ and its Borel subgroup $B^-_{n_{\ic}}$ consisting of lower triangular matrices.
We denote by $Y_{n_{\ic}}$ the Peterson variety in $\G{n_{\ic}}/B^-_{n_{\ic}}$.

\vspace{10pt}

\begin{corollary}\label{coro: a description for ZKJ}
Let $K\subseteq J\subseteq \dyn.$ Then, we have 
\begin{equation*}
\RS_{K,J}=\left\{x\dot{w}_JB^- \ \left| \ \text{\rm $x$ is $J$-Toeplitz and}~\begin{cases}\Delta_{i}(x\dot{w}_J)=0\quad\text{if}\quad i\in K\\
\Delta_{i}(x\dot{w}_J)\ne 0\quad\text{if}\quad i\notin K
\end{cases}
\right.
\right\}.
\end{equation*}
\end{corollary}

\begin{proof}
To begin with, we note that the right hand side of the desired equality is well-defined.
Denoting it by $\RS'_{K,J}$, we prove that $\RS_{K,J}=\RS'_{K,J}$.
First we show that 
\begin{equation}\label{RKJ sub RS'KJ}
\RS_{K,J}\subseteq \RS'_{K,J}.
\end{equation}
By \eqref{mathcal{X}J=J-toeplitz}, an arbitrary element $gB^-\in \RS_{K,J}=Y\cap \Omega_{w_K}^{\circ} \cap X_{w_J}^{\circ}$ can be written as  
\begin{equation*}
 gB^-=x\dot{w}_JB^-
\end{equation*}
for some $J$-Toeplitz matrix $x\in \G{n}$.
Since each block $x(\Jb{\ic})$ is a Toeplitz matrix in $\G{n_{\ic}}$, it follows from \eqref{mathcal{X}J=J-toeplitz} (for $\G{n_{\ic}}/B_{n_{\ic}}^-$) that \vspace{3pt}
\begin{equation}\label{eq: J-Toeplitz 10}
x(\Jb{\ic})\dot{w}_J(\Jb{\ic})B_{n_{\ic}}^- \in Y_{n_{\ic}}\cap (B^-_{n_{\ic}}\dot{w}_J(\Jb{\ic})B^-_{n_{\ic}}/B^-_{n_{\ic}}) 
\vspace{3pt}
\end{equation}
for $1\le \ic\le m$, 
where we note that $\dot{w}_J(\Jb{\ic})$ is the longest permutation matrix in $\G{n_{\ic}}$.
Since $x\dot{w}_JB^-(=gB^-) \in \Omega_{w_K}^{\circ}=B^+\dot{w}_KB^-/B^-$, there exist $b^-\in B^-$ 
and $b^+\in B^+$ 
such that \vspace{3pt}
\begin{equation*}
x\dot{w}_J b^- = b^+\dot{w}_K . \vspace{3pt}
\end{equation*}
Since $K\subseteq J$, $K$-block diagonal matrices are also $J$-block diagonal. In particular, $\dot{w}_K$ in this equality is $J$-block diagonal. 
By focusing on each $\Jb{\ic}\times \Jb{\ic}\ (\subseteq[n]\times [n])$ diagonal block of this equality, it follows that \vspace{3pt}
\begin{equation*}
x(\Jb{\ic})\dot{w}_J(\Jb{\ic}) B_{n_{\ic}}^- \in B_{n_{\ic}}^+\dot{w}_K(\Jb{\ic}) B_{n_{\ic}}^-/B_{n_{\ic}}^- 
\vspace{3pt}
\end{equation*}
for $1\le \ic\le m$.
Combining this and \eqref{eq: J-Toeplitz 10}, it follows that  $x(\Jb{\ic})\dot{w}_J(\Jb{\ic})B_{n_{\ic}}^-$ in $\G{n_{\ic}}/B^-_{n_{\ic}}$ belongs to a Richardson stratum in the big cell: \vspace{3pt}
\begin{equation}\label{eq: J-Toeplitz 30}
x(\Jb{\ic})\dot{w}_J(\Jb{\ic}) B_{n_{\ic}}^- \ \ \in \ \  
Y_{n_{\ic}}\cap (B_{n_{\ic}}^+\dot{w}_K(\Jb{\ic}) B_{n_{\ic}}^-/B_{n_{\ic}}^-)\cap (B^-_{n_{\ic}}\dot{w}_J(\Jb{\ic})B^-_{n_{\ic}}/B^-_{n_{\ic}})\vspace{3pt}
\end{equation}
for $1\le \ic\le m$.

We prove \eqref{RKJ sub RS'KJ} by using \eqref{eq: J-Toeplitz 30} in what follows.
For each $j\in J_{\ic}$, we set \vspace{3pt}
\begin{equation}\label{eq: def of j'}
j' \coloneqq j+1-\min J_{\ic} \in [n_{\ic} -1]\vspace{3pt}
\end{equation}
which encodes the position of $j$ in the set $J_{\ic}$; for example, if $j=\min J_{\ic}$, then $j'=1$.
Recall that $\dot{w}_J(\Jb{\ic})$ in \eqref{eq: J-Toeplitz 30} is the longest permutation matrix in $\G{n_{\ic}}$.
Applying Rietsch's result (Proposition~\ref{lemm: Rietsch's result for Z KI}) to the element in \eqref{eq: J-Toeplitz 30}, 
it follows that there exists a Toeplitz matrix $x_{\ic}\in\G{n_{\ic}}$ such that  \vspace{3pt}
\begin{itemize} \vspace{3pt}
\item[(1)] $x(\Jb{\ic})\dot{w}_J(\Jb{\ic})B_{n_{\ic}}^- = x_{\ic}\dot{w}_J(\Jb{\ic})B_{n_{\ic}}^-$ ; \vspace{5pt}
\item[(2)] For $j\in J_{\ic}$, we have $\Delta_{j'}^{(n_{\ic})}(x_{\ic}\dot{w}_J(\Jb{\ic}))=0 \  \text{if and only if} \ j\in K\cap J_{\ic}$,  \vspace{3pt}
\end{itemize}
where $j'$ is the one defined in \eqref{eq: def of j'} and $\Delta_{j'}^{(n_{\ic})}:\G{n_{\ic}}\rightarrow \C$ is the minor defined in \eqref{eq:4.3 first} for $\G{n_{\ic}}$.
In (1), both of $x(\Jb{\ic})$ and $x_{\ic}$ are lower triangular matrices having $1$'s on the diagonal, and $\dot{w}_J(\Jb{\ic})$ is the longest permutation matrix in $\G{n_{\ic}}$. Hence, (1) simply implies that $x(\Jb{\ic})=x_{\ic}$.
Thus, the condition (2) can be written as follows; for $j\in J_{\ic}$, we have \vspace{3pt}
\begin{equation*}
\Delta_{j'}^{(n_{\ic})}\big(x(\Jb{\ic})\dot{w}_J(\Jb{\ic})\big)=0\quad \text{if and only if}\quad j\in K\cap J_{\ic} . \vspace{3pt}
\end{equation*}
Since $x$ is a $J$-Toeplitz matrix, the determinants of the diagonal blocks of $x\dot{w}_J\in\G{n}$ are all equal to $1$, and hence the left hand side of this equality can be written as $\Delta_{j}(x\dot{w}_J)$, where $\Delta_{j}\colon \G{n}\rightarrow \C$ is the function defined in \eqref{eq:4.3 first} for $\G{n}$.
Namely, we have
\begin{equation*}
\Delta_{j}(x\dot{w}_J)=0\quad \text{if and only if}\quad j\in K\cap J_{\ic} .
\end{equation*}
This holds for each connected component $J_{\ic}$ $(1\le \ic\le m)$, and we have $\Delta_{j}(x\dot{w}_J)=1$ for all $j\notin J$ since $x$ is $J$-Toeplitz (see Example~\ref{ex: J-Toeplitz matrices}).
Therefore, we conclude that for $j\in \dyn$ we have
\begin{equation*}
\Delta_{j}(x\dot{w}_J)=0\quad \text{if and only if}\quad j\in K .
\end{equation*}
Recalling $gB^-=x\dot{w}_JB^-$ from the construction, we conclude that $\RS_{K,J}\subseteq \RS'_{K,J}$ as claimed in \eqref{RKJ sub RS'KJ}.

Conversely, let us show that
\begin{equation*}
\RS_{K,J}\supseteq \RS'_{K,J}.
\end{equation*}
Take an arbitrary element $x\dot{w}_J B^- \in \RS'_{K,J}$, where $x$ is a $J$-Toeplitz matrix. 
By \eqref{mathcal{X}J=J-toeplitz}, we have $x\dot{w}_J B^- \in Y\cap X_{w_J}^{\circ}$.
In what follows, let us prove that
\begin{equation}\label{eq: xwJ in OmegawJ}
x\dot{w}_J B^- \in \Omega_{w_K}^{\circ} 
\end{equation}
which completes the proof.
By the definition of $\RS'_{K,J}$, we have 
\begin{equation*}
\Delta_{j}(x\dot{w}_J)=0\quad \text{if and only if}\quad j\in K .
\end{equation*}
By reversing the argument above, we obtain the following claim for each $\Jb{\ic}\times \Jb{\ic}$ diagonal block; for $j\in J_{\ic}$, we have \vspace{3pt}
\begin{equation*}
\Delta_{j'}\big(x(\Jb{\ic})\dot{w}_J(\Jb{\ic})\big)=0
\quad \text{if and only if}\quad j\in K\cap J_{\ic} , \vspace{3pt}
\end{equation*}
where $j'\in [n_{\ic} -1]$ is the position of $j$ in $J_{\ic}$ defined in \eqref{eq: def of j'}.
This and Rietsch's result (Proposition~\ref{lemm: Rietsch's result for Z KI}) for $\G{n_{\ic}}/B_{n_{\ic}}^-$ imply that each $x(\Jb{\ic})\dot{w}_J(\Jb{\ic})B_{n_{\ic}}^-$ belongs to a Richardson stratum in the big cell \vspace{3pt}
\begin{equation*}
Y_{n_{\ic}}\cap (B_{n_{\ic}}^+\dot{w}_K(\Jb{\ic}) B_{n_{\ic}}^-/B_{n_{\ic}}^-)\cap (B^-_{n_{\ic}}\dot{w}_J(\Jb{\ic})B^-_{n_{\ic}}/B^-_{n_{\ic}}). \vspace{3pt}
\end{equation*}
In particular, we have 
\begin{equation*}
x(\Jb{\ic})\dot{w}_J(\Jb{\ic})B_{n_{\ic}}^- \in 
B_{n_{\ic}}^+\dot{w}_K(\Jb{\ic}) B_{n_{\ic}}^-/B_{n_{\ic}}^- . \vspace{3pt}
\end{equation*}
This means that there exist $b_{\ic}^-\in B^-_{n_{\ic}}$ and $b_{\ic}^+\in B^+_{n_{\ic}}$ such that \vspace{3pt}
\begin{equation*}
x(\Jb{\ic})\dot{w}_J(\Jb{\ic})b_{\ic}^- =
b_{\ic}^+ \dot{w}_K(\Jb{\ic}) . \vspace{3pt}
\end{equation*}
Let $b^-\in B^-$ be the $J$-block diagonal matrix having $b_{\ic}^-$ on its $\Jb{\ic}\times \Jb{\ic}\ (\subseteq[n]\times [n])$ diagonal block for $1\le \ic\le m$ and $1$'s on the diagonal blocks of size $1$. Similarly, define $b^+\in B^+$ from $b_{\ic}^+$ $(1\le \ic\le m)$ by the same manner.
Then we obtain an equality of $J$-block diagonal matrices
\begin{equation*}
x\dot{w}_Jb^- = b^+ \dot{w}_K
\end{equation*}
which can be shown by comparing the diagonal blocks (including the blocks of size $1$). Hence, we obtain that
\begin{equation*}
x\dot{w}_J B^- \in B^+\dot{w}_K B^-/B^- = \Omega_{w_K}^{\circ}  . \vspace{3pt}
\end{equation*}
as claimed in \eqref{eq: xwJ in OmegawJ}.
Thus, we conclude that $\RS'_{K,J}\subseteq \RS_{K,J}$ holds as well.
\end{proof}

\vspace{10pt}

\begin{example}\label{eg: RJK}
\rm
Let $n=5$. Then, for example, we have 
\begin{equation*}
\RS_{\{1\},\{1,2,\ \ ,4\}}
=\left\{
  \left.
  \left(
  \begin{array}{@{\,}ccc|ccc@{\,}}
      0&0&1&0&0\\
      0&-1&a&0&0\\
      1&-a&b&0&0\\ \hline
      0&0&0&0&1\\
      0&0&0&-1&x
  \end{array}
  \right)
  B^- \ \right| \ 
  \begin{matrix}
  \Delta_1=a^2-b= 0, \ \Delta_2=b\ne0, \\
  (\Delta_3=1\ne0), \ \Delta_4=x\ne0
  \end{matrix}
  \right\}.
\end{equation*}
\end{example}

\vspace{20pt}

\section{Totally nonnegative part of $Y$}\label{sec: nonnegative part of Y}
In this section, we study the totally nonnegative part of the Peterson variety with its cell decomposition introduced by Rietsch (\cite{ri03,ri06}).
We begin with reviewing Lusztig's theory of the totally nonnegative parts of $\G{n}$ and $\G{n}/B^-$ from \cite{lu94}.

\subsection{Totally nonnegative part of $\G{n}$}
Recall from Section~\ref{subsec: decomposition of Y} that we have
\begin{align*}
x_i(t)= \exp(te_i) \in U^+(\subseteq \G{n})
\quad \text{and} \quad
y_i(t)= \exp(tf_i) \in U^-(\subseteq \G{n})
\end{align*}
for $i\in \dyn$ and $t\in \C$.
Let $U^+_{\ge0}$ be the submonoid $($with $1$$)$ of $U^+$ generated by the elements $x_i(a)$ for $i\in \dyn$ and $a\in\mathbb{R}_{\ge 0}$.
Following \cite[Proposition 2.7]{lu94}, we decompose $U^+_{\ge0}$ as follows.
For each $w\in \mathfrak{S}_n$ and its reduced expression $w=s_{i_1}s_{i_2}\cdots s_{i_p}$, the map given by
\begin{equation*}
\mathbb{R}^p_{>0}\to U^+_{\ge0} \quad ; \quad (a_1,a_2,\cdots, a_p)\to x_{i_1}(a_1)x_{i_2}(a_2)\cdots x_{i_p}(a_p)
\end{equation*}
is known to be injective. Its image, denoted by $U^+(w)$, does not depend on the choice of the reduced expression $w=s_{i_1}s_{i_2}\cdots s_{i_p}$. Moreover, this provide us a disjoint decomposition
\begin{align}\label{eq: decomp of U^-}
U^+_{\ge 0}=\bigsqcup_{w\in \mathfrak{S}_n}U^+(w).
\end{align}
We set $U^+_{>0}\coloneqq  U^+(w_0)$,
where $w_0$ is the longest element in $\Sn$. 

\vspace{10pt}

\begin{example}\rm 
Let $n=3$. Then we have $w_0=s_1s_2s_1$ so that 
\begin{equation*}\begin{split}
U^+_{>0}&=U^+(s_1s_2s_1)
=\{x_1(a)x_2(b)x_1(c) \mid a>0, b>0, c>0\}\\
&=\left\{  \left. \begin{pmatrix}
      1 & a &0 \\
      0 & 1 &0\\
      0&0&1
   \end{pmatrix}\begin{pmatrix}
      1 & 0 &0 \\
      0 & 1 &b\\
      0&0&1
   \end{pmatrix}\begin{pmatrix}
      1 & c &0 \\
      0 & 1 &0\\
      0&0&1
   \end{pmatrix}\ \right| \ a>0, \ b>0, \ c>0\right\}\\
   &=\left. \left\{\begin{pmatrix}
      1 & x& y \\
      0 & 1 &z\\
      0&0&1
   \end{pmatrix}\ \right| \ x>0, \ y>0, z>0, \ xz-y>0\right\},
\end{split}
\end{equation*}
where we leave the reader to verify the last equality.
\end{example}

\vspace{10pt}

\begin{example}\label{ex: U plus wJ}\rm
Let $n=9$ and $J=\{1, 3,4, 7,8\}=\{1\}\sqcup\{3,4\}\sqcup\{7,8\}$.
Then the longest element $w_J$ of the Young subgroup $\mathfrak{S}_J$ is given by $w_J=s_1\cdot s_3s_4s_3\cdot s_7s_8s_7$. Hence, it follows that
\begin{equation*}
\begin{split}
U^+(w_J)&=U^+(s_1\cdot s_3s_4s_3\cdot s_7s_8s_7)\\
&=\left\{\left( \left.
 {\tiny
 \begin{array}{@{\,}cc|ccc|c|ccc@{\,}}
     1& \!\!a&  & & & & & &  \\
    &  \!\!1& & & & & & &  \\ \hline
    & & 1&  \!\!x & \!\!y & & & &  \\ 
     & &&\!\!1 &\!\!z & & & &   \\
     & && & \!\!1&  & & &\\ \hline
     & & & & & \!\!1\!\! & & &\\\hline
     & & &  & & & 1 & \!\!s & \!\!t\\ 
     & & & & & & & \!\!1 &\!\!u\\ 
     & & & & & & & & \!\!1\\
 \end{array}
 }
 \right) \ \right| \ 
 \begin{array}{l}
 a>0, \\
 x>0, \ y>0, \ z>0, \ xz-y>0, \\
 s>0, \ t>0, \ u>0, \ su-t>0
 \end{array}
 \right\}
\end{split}
\end{equation*}
by the previous example.
We note that each block of the matrix appearing in this set belongs $U^+_{\ge0}$ for special linear groups of smaller ranks.
\end{example}

\vspace{10pt}

Recall that we have the Bruhat decomposition of $\G{n}$:
\begin{align*}
 \G{n} = \bigsqcup_{w\in \Sn} B^-\dot{w}B^- ,
\end{align*}
where $\dot{w}\in\G{n}$ is the representative of $w$ defined in \eqref{eq: def of representative}.
Since $U^+_{\ge0}\subseteq \G{n}$, the intersections $U^+_{\ge0}\cap B^-\dot{w}B^-$ for $w\in W$ decompose $U^+_{\ge0}$. The following claim means that this coincides with the decomposition in \eqref{eq: decomp of U^-}.

\begin{lemma}\label{lem: plus w = bruhat}
For $w\in W$, we have $U^+(w) = U^+_{\ge0}\cap B^-\dot{w}B^-$. 
\end{lemma}

\begin{proof}
By \cite[Proposition~2.7~(d)]{lu94}, we have $U^+(w) \subseteq U^+_{\ge0}\cap B^-\dot{w}B^-$. In addition, both of these decompose $U^+_{\ge0}$:
\begin{align*}
 U^+_{\ge0} = \bigsqcup_{w\in \Sn} U^+(w)
 \quad \text{and} \quad
 U^+_{\ge0} = \bigsqcup_{w\in \Sn} (U^+_{\ge0}\cap B^-\dot{w}B^-).
\end{align*}
Thus, the claim follows.
\end{proof}

\vspace{10pt}

Similarly, $U^-_{\ge0}$ and $U^-(w)\subseteq U^-_{\ge0}$ for $w\in \mathfrak{S}_{n}$ are defined in the same manner, where we exchange $x_i(a)$ in the definitions of $U^+_{\ge0}$ and $U^+(w)$ to $y_i(a)$. In particular, we have 
\begin{align}\label{eq: minus decomp}
 U^-_{\ge0} = \bigsqcup_{w\in \Sn} U^-(w).
\end{align}
Similar to Lemma~\ref{lem: plus w = bruhat}, we have
\begin{align}\label{eq: minus w = bruhat}
 U^-(w) = U^-_{\ge0}\cap B^+\dot{w}B^+ 
\end{align}
for $w\in \Sn$. In fact, we can obtain this equality by taking the transpose to the equality of Lemma~\ref{lem: plus w = bruhat} for $w^{-1}$ (\cite[Section~1.2]{lu94}).

Also, $T_{>0}$ is defined as follows.
For $i\in[n]$, let $\alpha^{\vee}_i \colon \C^{\times}\rightarrow T\ (\subset \G{n})$ be the homomorphism which sends $t\in \C^{\times}$ to the element of $T$ which has $t$ and $t^{-1}$ on the $i$-th and $(i+1)$-st component, respectively, and $1$'s on the rest of the component:
\begin{align*}
 \alpha^{\vee}_i(t) =
 \text{diag}(1,\ldots,1,t,t^{-1},1\ldots,1)
 \in T .
\end{align*}
Then, $T_{>0}$ is defined to be the submonoid of $T(\subseteq \G{n})$ generated by the elements $\alpha^{\vee}_i(a)$ for $a\in\mathbb{R}_{>0}$ and $i\in \dyn$.
It is straightforward to verify that
\begin{align*}
 T_{>0} = \{\text{diag}(t_1,\ldots,t_n)\in T \mid t_i>0 \ \text{for $1\le i\le n$}\}.
\end{align*}

According to \cite{lu94}, the totally nonnegative part of $\G{n}$ is the submonoid of $\G{n}$ generated by the elements $x_i(a)$, $y_i(a)$ for $i\in \dyn$, $a\in\R_{\ge0}$ and by the elements $\alpha^{\vee}_i(t)$ for $i\in \dyn$, $t\in\R_{>0}$. It can be written as
\begin{align*}
 (\G{n})_{\ge0} = U^-_{\ge0}T_{>0}U^+_{\ge0} = U^+_{\ge0}T_{>0}U^-_{\ge0}.
\end{align*}
By A. Whitney's theorem (\cite{wh52}), we can also express $(\G{n})_{\ge0}$ as follows (e.g.\ see \cite{lu08}): 
\begin{align}\label{eq: type A nonnegativity}
 (\G{n})_{\ge0} = \{ g\in \G{n} \mid \text{all minors of $g$ are $\ge0$}\}.
\end{align}
The totally positive part of $\G{n}$ is defined to be
\begin{align*}
 (\G{n})_{>0} \coloneqq U^-_{>0}T_{>0}U^+_{>0} = U^+_{>0}T_{>0}U^-_{>0},
\end{align*}
and it follows that $(\G{n})_{\ge0}=\overline{(\G{n})_{>0}}$ (\cite[Theorem~4.3 and Remark~4.4]{lu94}).
An equality similar to \eqref{eq: type A nonnegativity} holds for $(\G{n})_{>0}$ as well by replacing the symbols $\ge$ to $>$ (see \cite{lu08,wh52}).

\vspace{10pt}

\subsection{Totally nonnegative part of the flag variety}
According to \cite{lu94}, the totally positive part of $\G{n}/B^-$ is defined to be
\begin{align*}
(\G{n}/B^-)_{> 0}
\coloneqq  
\{gB^- \in \G{n}/B^- \mid g\in (\G{n})_{>0}\} =\{uB^- \in \G{n}/B^- \mid u\in U^+_{>0}\}, 
\end{align*}
and the totally nonnegative part of $\G{n}/B^-$ is defined to be
\begin{align*}
(\G{n}/B^-)_{\ge 0}\coloneqq  \overline{(\G{n}/B^-)_{> 0}}, 
\end{align*}
where the closure is taken in $\G{n}/B^-$ with respect to the classical topology. 
For example, if $n=2$, then 
\begin{equation*}
 (\G{2}/B^-)_{> 0}= \{uB^- \in \G{2}/B^- \mid u\in U^+_{>0}\}
 =
 \left\{
 \left.
 \begin{pmatrix} 
  1 & t \\
  0 & 1 \\
 \end{pmatrix}B^- \ \right| \ t>0\right\}.
\end{equation*}
One can verify that its closure $(\G{2}/B^-)_{\ge 0}= \overline{(\G{2}/B^-)_{> 0}}$ is obtained by adding two points $B^-$ and $\dot{w}_0B^-$, and hence it is homeomorphic to an interval $[0,1]$.

\vspace{10pt}

Recall that we have $U^+(w) = U^+_{\ge0}\cap B^-\dot{w}B^-$ for $w\in \Sn$ from 
Lemma~\ref{lem: plus w = bruhat}. One can verify that this equality implies that
\begin{equation}\label{x lies in U-nonnegatiave}
U^+(w)B^-/B^- = (\G{n}/B^-)_{\ge0}\cap (B^+B^-/B^-)\cap (B^-\dot{w}B^-/B^-)
\end{equation}
for $w\in \Sn$ (e.g.\ \cite[Sect.\ 1.3]{ri99}) by a straightforward argument.
Also, in the proof of Corollary~10.5 of \cite{ri06}, 
Rietsch proved that
\begin{equation}\label{lemm: exptf acts on xwoB-}
U^-_{>0}\cdot(\G{n}/B^-)_{\ge 0}\subseteq (\G{n}/B^-)_{\ge 0}\cap (B^+B^-/B^-).
\end{equation}
We use these properties of $(\G{n}/B^-)_{\ge 0}$ to study the totally nonnegative part of the Peterson variety in the next subsection.

For the Borel subgroup $B^+\subset \G{n}$ consisting of upper triangular matrices, $(\G{n}/B^+)_{>0}$ and $(\G{n}/B^+)_{\ge0}$ are defined similarly by changing the role of $U^+_{>0}$ to $U^-_{>0}$.

\vspace{10pt}
\subsection{Totally nonnegative part of the Peterson variety}\label{subsec: A decompostion of Y ge 0}
We now study the totally nonnegative part of the Peterson variety introduced by Rietsch in \cite{ri03,ri06}.

\begin{definition}[\cite{ri03,ri06}]\label{definition of nonnegative part of Peterson variety}
 The totally nonnegative part of the Peterson variety $Y$ in $\G{n}/B^-$ is defined by $Y_{\ge 0}\coloneqq  Y\cap (\G{n}/B^-)_{\ge 0}$.
\end{definition}

\vspace{10pt}

Recall that we have the following decomposition of $Y$ by the Richardson strata in \eqref{decomposition of Y as union of ZKJ}:
\begin{equation*}
Y=\bigsqcup_{K\subseteq J\subseteq \dyn}\RS_{K,J} .
\end{equation*}

\begin{definition}[\cite{ri06}]\label{def: nonnegative part of ZKJ} 
For $K\subseteq J\subseteq \dyn$, we set
\begin{equation}\label{eq: def of YKJ>0}
\RS_{K,J;>0}\coloneqq \RS_{K,J}\cap(\G{n}/B^-)_{\ge 0}.
\end{equation}
\end{definition}

\vspace{10pt}

By construction, we obtain a stratification of $Y_{\ge 0}$:
\begin{equation}\label{decomposition of nonnegative part of Peterson}
Y_{\ge 0}=\bigsqcup_{K\subseteq J\subseteq \dyn}\RS_{K,J;>0}.
\end{equation}
The goal of this subsection is to give a concrete description of $\RS_{K,J;>0}$ (see Proposition~\ref{prop: description for R KJ>0} below).
We being with preparing some lemmas. 

\vspace{10pt}

Let $J\subseteq \dyn$, and take the decomposition $J = J_1 \sqcup J_2 \sqcup \cdots \sqcup J_{m}$ into the connected components.
Similar to the notations prepared for the proof of Corollary~\ref{coro: a description for ZKJ}, we consider the special linear groups $\G{n_{\ic}}$ with $n_{\ic}\coloneqq |\Jb{\ic}|$ and the subgroups $U^-_{n_{\ic}}\subseteq B^-_{n_{\ic}} \subseteq \G{n_{\ic}}$, where $U^-_{n_{\ic}}$ is the subgroup consisting of all the lower triangular matrices with $1$'s on the diagonal $(1\le k\le m)$. We also consider $U^+_{n_{\ic}}\subseteq B^+_{n_{\ic}}\subseteq \G{n_{\ic}}$ by exchanging the role of ``lower triangular" to ``upper triangular".
For a $J$-block matrix $g\in \G{n}$ (see Section~\ref{subsec: on elements in Y ge 0 cap XJ} for the definition of $J$-block matrices), we denote the $\Jb{\ic}\times \Jb{\ic}\ (\subseteq[n]\times [n])$ diagonal block of $g$ by $g(\Jb{\ic})$ for $1\le \ic\le m$.
We keep these notations in the rest of this section.

\begin{lemma}\label{lemm: if and only if lemma which is needed}
Let $J\subseteq \dyn$, and suppose that $x\in U^-$ is a $J$-block diagonal matrix. Then, we have
\begin{equation*}
x\in U^-_{\ge 0} \quad \text{if and only if} \quad
x(\Jb{\ic})\in (U_{n_{\ic}}^-)_{\ge 0}\ (1\le \ic\le m).
\end{equation*}
\end{lemma}
\begin{proof}
Suppose that $x\in U^-_{\ge0}$. The matrix $x$ has $1$'s on the diagonal, and it is a $J$-block diagonal matrix by the assumption. Thus, each diagonal block $x(\Jb{\ic})$ belongs to $\G{n_{\ic}}$. 
Hence, the Bruhat decomposition of $\G{n_{\ic}}$ implies that 
there exists a permutation $w_{\ic}\in\mathfrak{S}_{J_{\ic}}$ such that $x(\Jb{\ic})\in B^+_{n_{\ic}}\dot{w}_{\ic}B^+_{n_{\ic}}$.
This claim holds for each $1\le \ic\le m$. Since $x$ is a $J$-block diagonal matrix having $1$'s on the diagonal, it follows that
\begin{equation*}
 x\in B^+\dot{w}B^+ ,
\end{equation*}
where we set $w\coloneqq w_1\times\cdots\times w_m\in\mathfrak{S}_{J}=\mathfrak{S}_{J_1}\times \mathfrak{S}_{J_2}\times \cdots\times\mathfrak{S}_{J_m}\subseteq \Sn$.
Thus, we have
\begin{equation}\label{eq: x in UwJ}
 x \in U^-_{\ge0}\cap B^+\dot{w}B^+ = U^-(w)
\end{equation}
by \eqref{eq: minus w = bruhat}.
Since $w$ is an element of the Young subgroup $\mathfrak{S}_{J}$, a reduced expression of $w$ can be obtained by listing reduced expressions for $w_1,\ldots,w_m$ in order.
Hence, by the definition of $U^-(w)$, elements of $U^-(w)$ are $J$-block diagonal matrices having elements of $(U_{n_{\ic}}^-)_{\ge 0}$ on $\Jb{\ic}\times \Jb{\ic}$ diagonal blocks for $1\le \ic\le m$ and having $1$'s on the diagonal blocks of size $1$.
In particular, \eqref{eq: x in UwJ} implies that each block $x(\Jb{\ic})$ belongs to $(U_{n_{\ic}}^-)_{\ge 0}$.

Conversely, suppose that a $J$-block diagonal matrix $x\in U^-$ satisfies $x(\Jb{\ic})\in (U_{n_{\ic}}^-)_{\ge0}$ for $1\le \ic\le m$. 
Then we have $x(\Jb{\ic})\in U_{n_{\ic}}^-(w_k)$ for some $w_k\in \mathfrak{S}_{J_{\ic}}$ by \eqref{eq: minus decomp} for $(U_{n_{\ic}}^-)_{\ge0}$.
Let $w_k=s_{i_1}s_{i_2}\cdots s_{i_p}$ be a reduced expression of $w_k$ for some $i_1, i_2,\ldots,i_p\in J_{\ic}$.
Then, we can write
\begin{equation}\label{eq: reduced word for xJk}
x(\Jb{\ic})=x_{i_1}(a_1)x_{i_2}(a_2)\cdots x_{i_p}(a_p) \in U_{n_{\ic}}^-(w_k)
\end{equation}
for some $(a_1,a_2,\cdots, a_p)\in \mathbb{R}^p_{>0}$ by the definition of $ U_{n_{\ic}}^-(w_k)$.
To study the entire matrix $x\in U^-$, we set 
\begin{equation*}
w\coloneqq w_1\cdots w_m\in \mathfrak{S}_J=\mathfrak{S}_{J_1}\times \mathfrak{S}_{J_2}\times \cdots\times\mathfrak{S}_{J_m}\subseteq \mathfrak{S}_n .
\end{equation*}
By listing the above (reduced) expressions of $w_k$ for $1\le \ic\le m$ in order, we obtain a reduced expression of $w$. Now, \eqref{eq: reduced word for xJk} for $1\le \ic\le m$ and the definition of $U^-(w)$ imply that
$
x \in U^-(w)
$
since $x\in U^-$ is $J$-block diagonal.
Since $U^-(w)\subseteq U_{\ge 0}^-$ by definition, we obtain that $x\in U_{\ge 0}^-$.
\end{proof}

\vspace{10pt}

\begin{lemma}\label{lem: yw0=u+b}
For $x\in U^-$, we have 
\begin{equation*}
x\dot{w}_0B^-\in (\G{n}/B^-)_{>0}
\quad \text{if and only if} \quad
x\in U^-_{>0}.
\end{equation*}
\end{lemma}
\begin{proof}
By \cite[Theorem~8.7]{lu94}, 
the isomorphism
\begin{equation*}
\varphi \colon \G{n}/B^+\to \G{n}/B^- 
\quad ; \quad 
gB^+\mapsto g\dot{w}_0B^-
\end{equation*}
restricts to an isomorphism between $(\G{n}/B^+)_{>0}$ and $(\G{n}/B^-)_{>0}$. 
We use this observation to prove the claim of this Lemma.

Suppose that $x\in U^-_{>0}$. Then we have $xB^+\in (\G{n}/B^+)_{>0}$ by the definition of $(\G{n}/B^+)_{>0}$ which means that 
\begin{equation*}
x\dot{w}_0B^- =\varphi(xB^+) \in (\G{n}/B^-)_{>0}
\end{equation*}
by the above observation.
Conversely, suppose that $x\dot{w}_0B^-\in (\G{n}/B^-)_{>0}$. Then,
\begin{equation*}
xB^+ = \varphi^{-1}(x\dot{w}_0B^-)\in (\G{n}/B^+)_{>0}
\end{equation*}
by the above observation.
By the definition of $(\G{n}/B^+)_{>0}$, this means that there exists $x'\in U^-_{>0}$ such that $xB^+=x'B^+$.
Since $x,x'\in U^-$, we conclude that $x=x'\in U^-_{>0}$.
\end{proof}

\vspace{10pt}

\begin{proposition}\label{prop: x in U-ge 0}

Let $J\subseteq \dyn$, and suppose that $x\in U^-$ is a $J$-block diagonal matrix. Then, we have
\begin{equation*}
x\dot{w}_JB^-\in (\G{n}/B^-)_{\ge 0}
\quad \text{if and only if} \quad
x\in U^-_{\ge0}.
\end{equation*}
\end{proposition}

\begin{proof}
Suppose that $x\dot{w}_JB^-\in (\G{n}/B^-)_{\ge 0}$.
Let $t>0$ and set $y(t)\coloneqq\exp(tf)$, where $f$ is the regular nilpotent matrix defined in \eqref{eq: def of f}.
Then we have  $y(t)\in U^-_{>0}$ by \cite[Proposition 5.9 (b)]{lu94}.
Thus, \eqref{lemm: exptf acts on xwoB-} implies that
\begin{equation*}
y(t)\cdot x\dot{w}_JB^-\in (\G{n}/B^-)_{\ge0}\cap (B^+B^-/B^-).
\end{equation*}
Note that we also have
\begin{equation*}
y(t)\cdot x\dot{w}_JB^-\in B^-\dot{w}_JB^-/B^-
\end{equation*}
since $y(t)$ and $x$ belong to $B^-$.
Therefore, we obtain that
\begin{equation*}
y(t)\cdot x\dot{w}_JB^-\in 
(\G{n}/B^-)_{\ge0}\cap (B^+B^-/B^-)\cap (B^-\dot{w}_JB^-/B^-) 
=U^+(w_J)B^- ,
\end{equation*}
where the last equality follows from \eqref{x lies in U-nonnegatiave}. 
This means that there exists $u^+\in U^+(w_J)$ and $b^-\in B^-$ such that
\begin{equation*}
y(t)\cdot x\dot{w}_J=u^+b^-.
\end{equation*}
Notice that all the matrices in this equality are $J$-lower triangular (see Example~\ref{ex: U plus wJ} for the form of $u^+$).
Thus, by focusing on each $\Jb{\ic}\times \Jb{\ic}$ diagonal block, we obtain 
\begin{equation*}
\big(y(t)x\big)(\Jb{\ic})\cdot \dot{w}_J(\Jb{\ic})=u^+(\Jb{\ic})\cdot b^-(\Jb{\ic})
\end{equation*}
for $1\le \ic\le m$.
Here, we have $u^+(\Jb{\ic})\in (U_{n_{\ic}}^+)_{>0}$ since $u^+\in U^+(w_J)$ (see Example~\ref{ex: U plus wJ}). Hence, we obtain that
\begin{equation*}
\big(y(t)x\big)(\Jb{\ic})\cdot \dot{w}_J(\Jb{\ic}) B_{n_{\ic}}^-
\in (\G{n_{\ic}}/B_{n_{\ic}}^-)_{>0},
\end{equation*}
where we note that $\dot{w}_J(\Jb{\ic})$ is the longest permutation matrix in $\G{n_{\ic}}$.
Applying Lemma~\ref{lem: yw0=u+b} to this equality, we obtain that 
\begin{equation}\label{eq: ytx nonnegative}
\big(y(t)x\big)(\Jb{\ic})\in (U_{n_{\ic}}^-)_{>0}.
\end{equation}
Note that we have $\lim_{t\rightarrow+0}y(t)=1$.
Since $(U_{n_{\ic}}^-)_{\ge0}$ is a closed subset of $U_{n_{\ic}}^-$ containing $(U_{n_{\ic}}^-)_{>0}$ (\cite[Proposition 4.2]{lu94}), 
it follows that $x(\Jb{\ic})\in (U_{n_{\ic}}^-)_{\ge0}$ by taking $t\rightarrow+0$ in \eqref{eq: ytx nonnegative}. Therefore, we obtain $x\in U^-_{\ge0}$ by Lemma~\ref{lemm: if and only if lemma which is needed}.

Conversely, suppose that $x\in U^-_{\ge 0}$.
Let $t>0$, and set $y(t)=\exp(tf)\in U^-_{>0}$ as above. 
Since $x\in U^-_{\ge0}$, we have
$x(\Jb{\ic})\in (U_{n_{\ic}}^-)_{\ge 0}$ by Lemma~\ref{lemm: if and only if lemma which is needed}. 
Noticing that $y(t)(\Jb{\ic})=\exp(tf(\Jb{\ic}))\in (U_{n_{\ic}}^-)_{>0}$, we have
\begin{equation*}
 \big(y(t)(\Jb{\ic})\big)\cdot x(\Jb{\ic})\in (U_{n_{\ic}}^-)_{>0}
\end{equation*}
by \cite[Section~2.12]{lu94}.
By Lemma~\ref{lem: yw0=u+b}, we obtain that
\begin{equation*}
 \big(y(t)(\Jb{\ic})\big)\cdot x(\Jb{\ic})\cdot \dot{w}_J(\Jb{\ic}) \in (\G{n_{\ic}}/B_{n_{\ic}}^-)_{>0},
\end{equation*}
where $\dot{w}_J(\Jb{\ic})$ is the longest permutation matrix in $\G{n_{\ic}}$.
Namely, there exists $b_{\ic}^-\in B_{n_{\ic}}^-$ and $u_{\ic}^+\in (U_{n_{\ic}}^+)_{>0}$ such that 
\begin{equation*}
 \big(y(t)(\Jb{\ic})\big)\cdot x(\Jb{\ic})\cdot(\dot{w}_J)(\Jb{\ic})=u_{\ic}^+b_{\ic}^-.
\end{equation*}
Let $u^+\in U^+$ be the $J$-block diagonal matrix having $u_{\ic}^+$ for $\Jb{\ic}\times \Jb{\ic}$ diagonal block for $1\le \ic\le m$ and $1$'s on the diagonal blocks of size $1$.
Similarly, define $b^-\in B^-$ and $\widetilde{y}(t)\in U^+$ by the same manner by using $b_{\ic}^-$ and $y(t)(\Jb{\ic})$, respectively.
Then, we have 
\begin{equation*}
 \widetilde{y}(t)\cdot x\cdot \dot{w}_J  =u^+b^-.
\end{equation*}
Since $u^+\in U^+$ satisfies $u^+(\Jb{\ic})=u_{\ic}^+\in (U_{n_{\ic}}^+)_{\ge0}$ for all $1\le \ic\le m$, one can verify that 
\begin{equation*}
 u^+\in U^+_{\ge 0}\subseteq (\G{n})_{\ge0}
\end{equation*}
by an argument which is completely parallel to the proof of Lemma~\ref{lemm: if and only if lemma which is needed}.
Since we have $(\G{n})_{\ge0}=\overline{(\G{n})_{>0}}$ (\cite[Remark~4.4]{lu94}), 
the element $u^+$ can be written as the limit of a sequence in $(\G{n})_{>0}$.
Thus, we have 
\begin{equation*}
 \widetilde{y}(t)\cdot (x\dot{w}_J)B^-=u^+B^-\in(\G{n}/B^-)_{\ge 0}.
\end{equation*}
Therefore, we obtain that
\begin{equation*}
 x\dot{w}_JB^-=\lim\limits_{t\to +0}\big(\widetilde{y}(t)\cdot (x\dot{w}_J)B^-\big)\in(\G{n}/B^-)_{\ge 0}
\end{equation*}
since $(\G{n}/B^-)_{\ge 0}$ is a closed subset of $\G{n}/B^-$.
\end{proof}

\vspace{10pt}

We now give a concrete description of $\RS_{K,J;>0}$ defined in \eqref{eq: def of YKJ>0}.

\begin{proposition}\label{prop: description for R KJ>0}
Let $K\subseteq J\subseteq \dyn.$ Then, we have 
\begin{equation*}
\RS_{K,J;>0}=\left\{x\dot{w}_JB^- \ \left| \ 
\begin{matrix}
\text{
\rm $x\in U^-_{\ge0}$ is $J$-Toeplitz and}
~\begin{cases}\Delta_{i}(x\dot{w}_J)=0\quad\text{if}\quad i\in K\\
\Delta_{i}(x\dot{w}_J)> 0\quad\text{if}\quad i\notin K
\end{cases}
\end{matrix}
\right.
\right\}.
\end{equation*}
\end{proposition}

\begin{proof}
Denoting the right hand side of the desired equality by $\RS'_{K,J;>0}$, we show that $\RS_{K,J;>0}=\RS'_{K,J;>0}$. We first show that
\begin{align}\label{eq: first goal RKJ nonnegative}
\RS_{K,J;>0} \supseteq \RS'_{K,J;>0}.
\end{align}
By definition, an arbitrary element of $\RS'_{K,J;>0}$ can be written as $x\dot{w}_JB^-$ for some $J$-Toeplitz matrix $x\in U^-_{\ge0}$ such that
\begin{equation*}
~\begin{cases}\Delta_{i}(x\dot{w}_J)=0\quad\text{if}\quad i\in K\\
\Delta_{i}(x\dot{w}_J)> 0\quad\text{if}\quad i\notin K .
\end{cases}
\end{equation*}
By Corollary~\ref{coro: a description for ZKJ}, this means that $x\dot{w}_JB^-\in \RS_{K,J}$.
Since $x\in U^-_{\ge0}$ is $J$-Toeplitz, Proposition~\ref{prop: x in U-ge 0} implies that 
\begin{equation*}
x\dot{w}_JB^-\in (\G{n}/B^-)_{\ge 0}. 
\end{equation*}
Therefore, we conclude that $x\dot{w}_JB^-\in \RS_{K,J}\cap (\G{n}/B^-)_{\ge 0}=\RS_{K,J;>0}$ as claimed in \eqref{eq: first goal RKJ nonnegative}.

To complete the proof, we prove that the opposite inclusion holds in what follows.
Since $\RS_{K,J;>0}=\RS_{K,J}\cap (\G{n}/B^-)_{\ge0}$, 
Corollary~\ref{coro: a description for ZKJ} implies that an arbitrary element of $\RS_{K,J;>0}$ can be written as $x\dot{w}_JB^-$ for some $J$-Toeplitz matrix $x\in U^-$ such that
\begin{align*}
\text{$\Delta_{i}(x\dot{w}_J)=0$ if and only if $i\in K$.}
\end{align*}
Since $x\dot{w}_JB^-\in\RS_{K,J;>0}=\RS_{K,J}\cap (\G{n}/B^-)_{\ge0}$ and $x\in U^-$ is $J$-Toeplitz, it follows that $x\in U^-_{\ge0}$ by Proposition~\ref{prop: x in U-ge 0}.
Thus, to see that $x\dot{w}_JB^-\in \RS'_{K,J;>0}$, it suffices to show that $\Delta_{i}(x\dot{w}_J)\ge0$ for all $i\in \dyn$.
For this purpose, let us make an observation. For a matrix $g\in \G{n}$, the function $\Delta_i(g)$ was defined to be the $(n-i)\times (n-i)$ lower right minor of $g$ (in Section~\ref{subsec: description of RKJ}).
From this, one can verify that $\Delta_i(g\dot{w}_0)$ is the $(n-i)\times (n-i)$ lower \textit{left} minor of $g$.

Based on this observation, we now prove that $\Delta_{i}(x\dot{w}_J)\ge 0$ for all $i\in \dyn$. We take cases. First, suppose that $i\in J$. In this case, we have $i\in J_{\ic}$ for some $1\le \ic\le m$. Set $i'\coloneqq i+1-\min J_{\ic}$ which encodes the position of $i$ in $J_{\ic}$ (as in \eqref{eq: def of j'}).
This notation allows us to express the minor $\Delta_i(x\dot{w}_J)$ in terms of minors of diagonal blocks $(x\dot{w}_J)(\Jb{\ic})$:
\begin{align*}
\Delta_i(x\dot{w}_J)&=\Delta_{i'}\big((x\dot{w}_J)(\Jb{\ic})\big)
=\Delta_{i'}\big(x(\Jb{\ic})\cdot\dot{w}_J(\Jb{\ic})\big),
\end{align*}
where we note that $\dot{w}_J(\Jb{\ic})$ is the longest permutation matrix in $\G{n_{\ic}}$.
By applying the above observation to $\G{n_{\ic}}$, it follows that the right most expression in this equality is a lower left minor of $x(\Jb{\ic})$ which is a minor of $x$.
Since we saw $x\in U^-_{\ge 0}\subseteq (\G{n_{\ic}})_{\ge0}$ above, Whitney's description \eqref{eq: type A nonnegativity} for $(\G{n_{\ic}})_{\ge0}$ implies that $\Delta_i(x\dot{w}_J)$ (which is a minor of $x$ as we just saw) is nonnegative.

If $i\notin J$, then we have $\Delta_i(x\dot{w}_J)=1\ge0$ since $x\dot{w}_J$ is a $J$-block diagonal matrix having $1$'s on the diagonal blocks of size $1$. Hence, the claim follows.
\end{proof}

\vspace{10pt}

\begin{example}\label{eg: RJK nonnegative}
\rm
Let $n=5$. We computed $\RS_{\{1\},\{1,2,\ \ ,4\}}$ in Example~\ref{eg: RJK}.
By Proposition~\ref{prop: description for R KJ>0} with Lemma~\ref{lemm: if and only if lemma which is needed} in mind, an arbitrary element of $\RS_{\{1\},\{1,2,\ \ ,4\};>0}$ can be written as
\begin{equation*}
  \left(
  \begin{array}{@{\,}ccc|ccc@{\,}}
      0&0&1&0&0\\
      0&-1&a&0&0\\
      1&-a&b&0&0\\ \hline
      0&0&0&0&1\\
      0&0&0&-1&x
  \end{array}
  \right)
\end{equation*}
for some $a,b,x\in\R_{\ge0}$ such that $\Delta_1=a^2-b= 0, \ \Delta_2=b>0, \ 
  (\Delta_3=1), \ \Delta_4=x>0$.
\end{example}

\vspace{20pt}

\section{A morphism from $Y$ to $ X(\fan)$}
In \cite{ab-ze23}, the authors constructed a particular morphism from the Peterson variety $Y$ to the toric orbifold $X(\fan)$. 
In this section, we review its construction since we use it to prove Rietsch's conjecture on $Y_{\ge0}$ in the next section. Details can be found in \cite[Sect.~6]{ab-ze23}.

We keep the notations from Section~\ref{sec: notations} where we defined the simple roots $\alpha_i \colon T\rightarrow \C^{\times}$ and the fundamental weights $\varpi_i \colon T\rightarrow \C^{\times}$ for $i\in\dyn$.
By composing with the canonical projection $B^-\onto T$, we may regard them as a homomorphisms from $B^-$, and we denote them by the same symbols $\alpha_i$ and $\varpi_i$ by abusing notations.

\subsection{Two functions}
For $i\in\dyn$, we considered in Section~\ref{subsec: description of RKJ} the function
\begin{equation*}
\Delta_{i} \colon \G{n}\to\C
\end{equation*}
which takes the lower right $(n-i)\times (n-i)$ minor of a given matrix in $\G{n}$ (see the examples in \eqref{eq:4.3 second}).
For $u\in U^+$ and $b=(b_{i,j})\in B^-$, we have
\begin{equation}\label{eq:4.3}
 \Delta_i(ub)=\Delta_i(b)=b_{i+1,i+1}b_{i+2,i+2}\cdots b_{n,n}=(-\varpi_{i})(b),
\end{equation}
where we write $(-\varpi_i)(b)\coloneqq \varpi_i(b)^{-1}$ for $b\in B^-$. 
Since $U^+B^-=B^+B^-\subseteq \G{n}$ is a Zariski open subset, this means that $\Delta_i$ is the unique regular function on $\G{n}$ satisfying the above equality. In particular, $\Delta_i$ is an extension of the character $-\varpi_{i}$ on $B^-$.
Moreover, the function $\Delta_{i}\colon \G{n}\to\C$ is $B^-$-equivariant in the following sense:
\begin{equation*}
\Delta_{i}(gb)
=(-\varpi_{i})(b)\Delta_{i}(g)
\qquad (g\in \G{n}, \ b\in B^-).
\end{equation*}

Let $P_Y$ be the restriction of the principal $B^-$-bundle $\G{n}\to \G{n}/B^-$ on the Peterson variety $Y$;
\begin{equation*}
P_Y \coloneqq \{g\in \G{n} \mid gB^-\in Y\}=\{g\in \G{n} \mid (\Ad{g^{-1}} f)_{i,j}=0 \ (j>i+1)\},
\end{equation*}
where $f$ is the regular nilpotent matrix given in \eqref{eq: def of f}, and $(\Ad{g^{-1}}f)_{i,j}$ denotes the $(i,j)$-th component of the matrix $\Ad{g^{-1}}f$.
By definition, $P_Y\subseteq G$ is preserved by the multiplication of elements of $B^-$ from the right, and hence $P_Y$ inherits a right $B^-$-action. With this action, $P_Y$ is a principal $B^-$-bundle over $Y$ in the 
sense that the projection map $P_Y\rightarrow Y$ is $B^-$-equivariantly locally trivial.
Now consider a function
\begin{equation*}
q_i \colon P_Y\to\C;\quad g\mapsto -(\Ad{g^{-1}}f)_{i,i+1}.
\end{equation*}
This function is $B^-$-equivariant in the following sense; for $g\in P_Y$ and $b\in B^-$, we have 
\begin{equation}\label{eq:4.5}
q_i(gb)
= -\big(\Ad{b^{-1}}(\Ad{g^{-1}}f)\big)_{(i,i+1)}
= (-\alpha_i)(b)q_i(g),
\end{equation}
where the last equality follows since $gB\in Y$ implies that $(\Ad{g^{-1}} f)_{i,j}=0$ for $j>i+1$.
Here, we also write $(-\alpha_i)(b)\coloneqq \alpha_i(b)^{-1}$ for $b\in B^-$. 

These two functions $\Delta_i$ and $q_i$ give rise to sections of certain line bundles on $Y$ (\cite[Section~5]{ab-ze23}), and the following proposition means that their zero loci do not intersect for each $i\in\dyn$.

\begin{proposition}[{\cite[Proposition~5.6]{ab-ze23}}]\label{prop:Delta i intersects ad alpha i=emptyset}
For $i\in \dyn$, we have
\begin{equation*}
\{gB^-\in Y \mid \Delta_{i}(g)=0 \ \text{and}\ q_i(g)=0\}=\emptyset.
\end{equation*}
\end{proposition}

\begin{proof}
Recall that we defined the Peterson variety $Y$ in $\G{n}/B^-$ 
(Definition~\ref{definition of complex Peterson variety}).
The roots $-\alpha_1,\ldots,-\alpha_{n-1}$ can be thought as the set of simple roots for $B^-$, and then the matrix $f$ (which we used to define $Y$) is a sum of simple root vectors for $B^-$.
Also, $-\varpi_1,\ldots,-\varpi_{n-1}$ are the fundamental weights of $B^-$ with respect to $-\alpha_1,\ldots,-\alpha_{n-1}$. 
Now, in the language of \cite[Sect.~5]{ab-ze23}, the functions $\Delta_{i}$ and $q_i$ are written as $\Delta_{-\varpi_i}$ and $q_{-\alpha_i}$, respectively.
Therefore, the claim follows from \cite[Proposition~5.6]{ab-ze23} by taking $B^-$ as our choice of a Borel subgroup of $\G{n}$.
\end{proof}

\begin{remark}
{\rm
For an elementary proof of Proposition~\ref{prop:Delta i intersects ad alpha i=emptyset}, see \cite[Lemma~3.9]{AHKZ}.
}
\end{remark}

\vspace{10pt}

\subsection{Construction of the morphism $\Psi$}
Recall we have $X(\fan)=(\C^{\ddyn}-\EL)/T$ from Section~\ref{sec: Toric orbifold}.
Recall also that $P_Y$ is the restriction of the principal $B^-$-bundle $\G{n}\to \G{n}/B^-$ on the Peterson variety $Y$.
Now consider a morphism $\widetilde{\Psi} \colon P_Y\to\C^{\ddyn}$
given by
\begin{equation*}
\widetilde{\Psi}(g)=(\Delta_{1}(g),\dots, \Delta_{n-1}(g); q_{1}(g),\dots, q_{n-1}(g))
\end{equation*}
for $g\in P_Y$. Proposition~\ref{prop:Delta i intersects ad alpha i=emptyset} implies that $\Delta_{i}(g)$ and $q_i(g)$ can not be zero simultaneously for each $i\in \dyn.$ Thus, the image of $\widetilde{\Psi}$ lies in $\C^{\ddyn}- \EL$ which was defined in \eqref{eq: def of C-E}. Namely, we have
\begin{equation*}
\widetilde{\Psi} \colon P_Y\to \C^{\ddyn}- \EL.
\end{equation*}
Combining this with the quotient morphism $\C^{\ddyn}- \EL\rightarrow (\C^{\ddyn}- \EL)/T$, we obtain a morphism $P_Y\rightarrow X(\fan)=(\C^{\ddyn}- \EL)/T$.
Recall that $P_Y$ admits the multiplication of $B^-$ from the right. The $B^-$-equivariance (\ref{eq:4.3}) and (\ref{eq:4.5}) now imply that the resulting morphism $P_Y\rightarrow X(\fan)$ is $B^-$-invariant, where we note that $(-\varpi_{i})(t)=\varpi_i(t^{-1})$ and $(-\alpha_{i})(t)=\alpha_i(t^{-1})$ for $t\in T$.
Since the map $P_Y \rightarrow Y$ is $B^-$-equivariantly locally trivial, we obtain a morphism
\begin{equation*}
\Psi \colon Y\to X(\fan)
\end{equation*}
given by 
\begin{equation*}
\Psi(gB^-)=\big[\Delta_{1}(g),\dots, \Delta_{n-1}(g); q_{1}(g),\dots, q_{n-1}(g)\big]
\end{equation*}
for $gB^-\in Y$.

\vspace{10pt} 

\begin{remark}
{\rm
The toric orbifold $X(\fan)$ was defined as the quotient of $\C^{\ddyn}- \EL$ by the linear $T$-action given in \eqref{eq: def of T-action on C2r}. Since we have
\begin{align*}
(-\varpi_{i})(t)=\varpi_i(t^{-1}),\quad
(-\alpha_{i})(t) = \alpha_i(t^{-1})
\qquad (t\in T)
\end{align*}
for $i\in \dyn$, it follows that $X(\fan)$ can also be constructed from the linear $T$-action on $\C^{\ddyn}-\EL$ defined by the weights
$
-\varpi_1,\ldots,-\varpi_{n-1}, -\alpha_1,\ldots,-\alpha_{n-1}
$
which are fundamental weights and simple roots for the Borel subgroup $B^-\subseteq \G{n}$.
Therefore, the above morphism $\Psi\colon Y\rightarrow X(\fan)$ agrees with the one constructed in \cite{ab-ze23} by taking $B^-$ as our choice of a Borel subgroup of $\G{n}$.
}
\end{remark}

\vspace{20pt}

\section{Relation between $Y_{\ge 0}$ and $X(\fan)_{\ge0}$}

In this section, we show that $Y_{\ge 0}$ is homeomorphic to $X(\fan)_{\ge0}$
under the map $\Psi\colon Y\to X(\fan)$ given in the last section.

For each subset $J\subseteq\dyn$, we have the decomposition $J = J_1 \sqcup J_2 \sqcup \cdots \sqcup J_{m}$ into the connected components. 
In the rest of this section, we use the same notations $\Jb{\ic}$, $n_{\ic}=|\Jb{\ic}|$, $\G{n_{\ic}}$, $x(\Jb{\ic})$ etc for $1\le \ic\le m$ prepared to prove Corollary~\ref{coro: a description for ZKJ}.

\subsection{Restriction of $\Psi \colon Y\rightarrow X(\fan)$ to nonnegative parts}
We first prove that the map $\Psi \colon Y\rightarrow X(\fan)$ restricts to a map from $Y_{\ge0}$ to $X(\fan)_{\ge0}$.
To this end, recall that we have the decompositions
\begin{equation}\label{eq: two decompositions}
Y_{\ge 0}=\bigsqcup_{K\subseteq J\subseteq \dyn}\RS_{K,J;>0}
\quad \text{and} \quad
X(\fan)_{\ge0} = \bigsqcup_{K\subseteq J \subseteq \dyn} \Xp{K,J} 
\end{equation}
from \eqref{decomposition of nonnegative part of Peterson} and \eqref{eq: decomp of X nonnegative KJ}.
We prove that $\Psi$ sends each stratum $\RS_{K,J;>0}$ to $\Xp{K,J}$ in what follows.

\begin{lemma}\label{lemm: ad part for xwJB-}
Let $J\subseteq\dyn$.
If $x\in \G{n}$ is J-Toeplitz then 
\begin{equation*}
q_i(x\dot{w}_J)=-\big(\Ad{(x\dot{w}_J)^{-1}} f\big)_{i,i+1}=\begin{cases}
0\quad\text{if\quad$i\notin J$}\\
1\hspace{4mm}\text{if\quad$i\in J$},
\end{cases}
\end{equation*}
where $f$ is the regular nilpotent matrix given in \eqref{eq: def of f}.
\end{lemma}
\begin{proof}
For a matrix $g\in \G{n}$, one can verify that 
\begin{equation}\label{eq: w0gw0 obs}
 (\dot{w}_0^{-1}g\dot{w}_0)_{i,i+1} = -g_{n+1-i,n-i}
 \quad \text{for $i\in [n-1]$}
\end{equation}
by a direct calculation since $(\dot{w}_0)_{i,j}=\delta_{i,n+1-j}(-1)^{i-1}$ for $i,j\in\dyn$ and $\dot{w}_0^{-1}$ is the transpose of $\dot{w}_0$.
Based on this observation, we now prove the claim.

\vspace{5pt}

Case 1: We first consider the case $i\in J$. In this case, we have $i\in J_{k}$ for some $1\le \ic\le m$.
Then we have $i,i+1\in \Jb{\ic}$ by the definition of $\Jb{\ic}$ (see \eqref{def: definition of J'i} for the definition of $\Jb{\ic}$). 
Set $i'\coloneqq i+1-\min J_{\ic}$ which encodes the position of $i$ in $J_{\ic}$ (as in \eqref{eq: def of j'}).
Then we have $i',i'+1\in [n_{\ic}]$, and hence
\begin{equation*}
\begin{split}
\big(\Ad{(x\dot{w}_J)^{-1}} f\big)_{i,i+1}
&=\Big( \big(\Ad{(x\dot{w}_J)^{-1}}f\big)(\Jb{\ic}) \Big)_{i',i'+1} \\
&=\Big( \left(\dot{w}_J^{-1} x^{-1}fx\dot{w}_J\right)(\Jb{\ic}) \Big)_{i',i'+1} .
\end{split}
\end{equation*}
Since $\dot{w}_J$ and $\dot{w}_J^{-1}$ are $J$-block diagonal matrices, we can decompose the products in the last expression:
\begin{equation*}
\begin{split}
\big(\Ad{(x\dot{w}_J)^{-1}} f\big)_{i,i+1}
&=\Big( \dot{w}_J(\Jb{\ic})^{-1}\cdot(x^{-1}fx)(\Jb{\ic})\cdot\dot{w}_J(\Jb{\ic}) \Big)_{i',i'+1}.
\end{split}
\end{equation*}
Thus, by \eqref{eq: w0gw0 obs}, we obtain 
\begin{equation*}
\begin{split}
\big(\Ad{(x\dot{w}_J)^{-1}} f\big)_{i,i+1}
=&-\left((x^{-1}fx)(\Jb{\ic})\right)_{n_k+1-i',n_k-i'}\\
&=-\sum_{\ell=2}^{n_k}\big(x(\Jb{\ic})^{-1}\big)_{n_k+1-i',\ell}\big(x(\Jb{\ic})\big)_{\ell-1,n_k-i'}
\quad \text{(by \eqref{eq: def of f})}.
\end{split}
\end{equation*}
Since $x(\Jb{\ic})^{-1}$ and $x(\Jb{\ic})$ are lower triangular matrices, the summands in the last expression vanish except for the term $\ell = n_k+1-i'$.
In addition, the term $\ell = n_k+1-i'$ appears in the summation since $i'\in [n_{\ic} -1]$ implies that $n_k+1-i'\ge2$.
Thus, we obtain
\begin{equation*}
\begin{split}
\big(\Ad{(x\dot{w}_J)^{-1}} f\big)_{i,i+1}
&=-\big(x(\Jb{\ic})^{-1}\big)_{n_k+1-i',n_k+1-i'}\big(x(\Jb{\ic})\big)_{n_k-i',n_k-i'}\\
&=-1
\end{split}
\end{equation*}
since $x(\Jb{\ic})^{-1}$ and $x(\Jb{\ic})$ have $1$'s on the diagonal.
Hence we obtain $-\big(\Ad{(x\dot{w}_J)^{-1}} f\big)_{i,i+1}=1$ in this case.\vspace{5pt}

Case 2: We next consider the case $i\notin J$. 
Since the matrix $(x\dot{w}_J)^{-1}f(x\dot{w}_J)$ is $J$-lower triangular, the condition $i\notin J$ implies that we have
\begin{equation*}
\big(\Ad{(x\dot{w}_J)^{-1}} f\big)_{i,i+1}=\big((x\dot{w}_J)^{-1}f(x\dot{w}_J)\big)_{i,i+1} =0
\end{equation*}
in this case. This completes the proof.
\end{proof}

\vspace{10pt}

\begin{lemma}\label{lemma: image of nonnegative part of ZKJ}
Let $K\subseteq J\subseteq \dyn$. Then we have 
$
\Psi(\RS_{K,J;>0})\subseteq \Xp{K,J}.
$
\end{lemma}
\begin{proof}
By Proposition~\ref{prop: description for R KJ>0}, an arbitrary element of $\RS_{K,J;>0}$ can be written as $x\dot{w}_JB^-$ for some $J$-Toeplitz matrix $x\in U^-_{\ge0}$ such that
\begin{equation*}
\begin{cases}\Delta_{i}(x\dot{w}_J)=0\quad\text{if}\quad i\in K\\
\Delta_{i}(x\dot{w}_J)>0 \quad\text{if}\quad i\notin K .
\end{cases}
\end{equation*}
Since $x$ is $J$-Toeplitz, we have from Lemma~\ref{lemm: ad part for xwJB-} that
\begin{equation*}
q_i(x\dot{w}_J)
=\begin{cases}
0\quad\text{if $i\notin J$}\\
1\hspace{4mm}\text{if $i\in J$} .
\end{cases}
\end{equation*}
Thus, by Proposition~\ref{prop: nonnengative stratum}, $\Psi(x\dot{w}_JB^-)$ is an element of $\Xp{K,J}$.
\end{proof}

\vspace{10pt}

Now the following claim follows from Lemma~\ref{lemma: image of nonnegative part of ZKJ} and the decompositions \eqref{eq: two decompositions}.

\begin{proposition}\label{prop: resctriction to nonnegative}
The morphism $\Psi \colon Y\to X(\fan)$ restricts to a continuous map
\begin{equation*}
\Psi_{\ge 0} \colon Y_{\ge 0}\to X(\fan)_{\ge 0}
\end{equation*}
which sends $\RS_{K,J;>0}$ to $\Xp{K,J}$ for $K\subseteq J \subseteq \dyn$.
\end{proposition}

\vspace{10pt}

\subsection{Properties of $\Psi_{\ge 0}$}
In this subsection, we prove that $\Psi_{\ge 0} \colon Y_{\ge 0}\to X(\fan)_{\ge 0}$ is a homeomorphism. For that purpose, we apply the following result due to Rietsch (\cite{ri03}). 
For $i\in\dyn$, recall that $\Delta_i(g)$ is the lower right $(n-i)\times (n-i)$ minor of $g\in\G{n}$ (see \eqref{eq:4.3 second}).
By a direct calculation, one can verify that $\Delta_i(g\dot{w}_0)$ is the lower \textit{left} $(n-i)\times (n-i)$ minor of $g$. 

\begin{proposition}[{\cite[Proposition~1.2]{ri03}}]\label{prop: Rietsch's result}
Let $X_{\ge0}$ be the set of totally nonnegative Toeplitz matrices in $U^-$. Then, the map 
\begin{equation*}
X_{\ge0}\rightarrow \R^{n-1}_{\ge0}
\quad ; \quad
x\mapsto (\Delta_1(x\dot{w}_0),\ldots,\Delta_{n-1}(x\dot{w}_0))
\end{equation*}
is a homeomorphism, where each $\Delta_i(x\dot{w}_0)$ is the lower left $(n-i)\times (n-i)$ minor of $x$.
\end{proposition}

\vspace{10pt}

We apply this result as a key fact to show that $\Psi_{\ge 0} \colon Y_{\ge 0}\to X(\fan)_{\ge 0}$ is a homeomorphism.

\begin{proposition}\label{lemm: injection on ZKJ>0}
For $K\subseteq J\subseteq \dyn$, the map \begin{equation*}\Psi_{\ge 0} \colon \RS_{K,J;>0}\to \Xp{K,J}\end{equation*} is injective.
\end{proposition}

\begin{proof}
Take $p,p'\in\RS_{K,J;>0}$ such that $\Psi_{\ge 0}(p)=\Psi_{\ge 0}(p').$
By Proposition~\ref{prop: description for R KJ>0}, we may write $$\text{$p=x\dot{w}_JB^-$, $p'=x'\dot{w}_JB^-$ for some $J$-Toeplitz matrices 
$x,x'\in U^-_{\ge0}$}$$
such that
\begin{equation}\label{eq: with positive properties}
\begin{cases}
\Delta_{i}(x\dot{w}_J)=0\quad\text{if}\quad i\in K\\
\Delta_{i}(x\dot{w}_J)>0 \quad\text{if}\quad i\notin K
\end{cases}
\quad \text{and} \quad
\begin{cases}
\Delta_{i}(x'\dot{w}_J)=0\quad\text{if}\quad i\in K\\
\Delta_{i}(x'\dot{w}_J)>0 \quad\text{if}\quad i\notin K .
\end{cases}
\end{equation}
To prove that $p=p'$, we show that $x=x'$ in what follows.

The equality $\Psi_{\ge 0}(p)=\Psi_{\ge 0}(p')$ in $X(\fan)=(\C^{\ddyn}- \EL)/T$ implies that there exists $t\in T$ such that 
\begin{equation}\label{eq: delta and ad for x and x'}
\begin{split}
&\varpi_i(t)\cdot \Delta_{i}(x\dot{w}_J)=\Delta_{i}(x'\dot{w}_J), \\
&\alpha_i(t)\cdot q_i(x\dot{w}_J) =q_i(x'\dot{w})
\end{split}
\end{equation}
for $1\le i\le n-1$. 
Since $x,x'$ are $J$-Toeplitz matrices, to prove that $x=x'$, it suffices to show that
\begin{equation}\label{eq: x=x' on Jk}
x(\Jb{\ic})=x'(\Jb{\ic})
\qquad (1\le \ic\le m).
\end{equation}
Let us prove this in the rest of the proof.

By Lemma~\ref{lemm: ad part for xwJB-} and (\ref{eq: delta and ad for x and x'}), we have $\alpha_i(t)=1$ for $i\in J_{\ic}$, so the $\Jb{\ic}\times \Jb{\ic}$ diagonal block $t(\Jb{\ic})$ of the matrix $t$ takes of the form
\begin{equation*}
t(\Jb{\ic})=
\text{diag}(\lambda_k,\lambda_k,\ldots,\lambda_k)
\end{equation*}
for some $\lambda_k\in\C^{\times}$. Since $x$ and $x'$ are $J$-Toeplitz matrices, we have $\Delta_{n-i}(x\dot{w}_J)=\Delta_{n-i}(x'\dot{w}_J)=1$ for $i\notin J$. Applying this to the first equality in \eqref{eq: delta and ad for x and x'}, we obtain
\begin{equation}\label{eq: pi=1 on not J}
\varpi_i(t)=1
\qquad (i\notin J).
\end{equation}
Thus, we may prove inductively that
\begin{equation*}
\det(t(\Jb{\ic}))=(\lambda_k)^{n_k}=1
\qquad (1\le \ic\le m).
\end{equation*}
In particular, we obtain that $|\lambda_k|=1$ for $1\le \ic\le m$.
This and \eqref{eq: pi=1 on not J} imply that
\begin{equation}\label{eq: abs pi}
|\varpi_i(t)|=1
\qquad (i\in\dyn).
\end{equation}
We clam that
\begin{equation}\label{eq: injectivity 10'}
\text{$\Delta_{i}(x\dot{w}_J)=\Delta_{i}(x'\dot{w}_J)$}
\qquad (i\in J_{\ic}).
\end{equation}
To see this, we take cases.
If $J_{\ic}\subseteq K$, then the claim follows from \eqref{eq: with positive properties} since both sides of this equality are zero.
If $J_{\ic}\not\subseteq K$, then by \eqref{eq: with positive properties}, we have 
\begin{equation}\label{eq: JcapK and J-K}
\begin{split}
&\Delta_{i}(x\dot{w}_J)=\Delta_{i}(x'\dot{w}_J)=0 \quad (i\in J_{\ic}\cap K), \\
&\text{$\Delta_{i}(x\dot{w}_J)>0$ and $\Delta_{i}(x'\dot{w}_J)>0$} \quad (i\in J_{\ic}- K).
\end{split}
\end{equation}
Applying the latter claim to \eqref{eq: delta and ad for x and x'}, we see that $\varpi_i(t)>0$ for $i\in J_{\ic}- K$. Now \eqref{eq: abs pi} implies that $\varpi_i(t)=1$ for $i\in J_{\ic}- K$. Applying this to \eqref{eq: delta and ad for x and x'} again, we obtain that 
\begin{equation*}
\text{$\Delta_{i}(x\dot{w}_J)=\Delta_{i}(x'\dot{w}_J)$}\quad (i\in J_{\ic}- K).
\end{equation*}
Combining this with the former equality in \eqref{eq: JcapK and J-K}, we conclude \eqref{eq: injectivity 10'} as claimed above. 

We now prove \eqref{eq: x=x' on Jk} by using \eqref{eq: injectivity 10'}.
For each $i\in J_k$, we set $i'\coloneqq i+1-\min J_{\ic}$ which encodes the position of $i$ in $J_{\ic}$ (as in \eqref{eq: def of j'}).
Since $x$ and $x'$ in \eqref{eq: injectivity 10'} are $J$-Toeplitz matrices, it follows that
\begin{equation*}
\Delta_{i'}^{(n_k)}\big(x(\Jb{\ic})\dot{w}_J(\Jb{\ic})\big)=\Delta_{i'}^{(n_k)}\big(x'(\Jb{\ic})\dot{w}_J(\Jb{\ic})\big)
\qquad (i'\in [n_k]),
\end{equation*}
where $\Delta_{i'}^{(n_k)}\colon \G{n_{\ic}}\rightarrow \C$ is the function defined in \eqref{eq:4.3 first} for $\G{n_{\ic}}$, and $\dot{w}_J(\Jb{\ic})$ is the longest permutation in $\G{n_{\ic}}$.
In this equality, we have
\begin{equation*}\label{eq: nonnegative on Jk}
x(\Jb{\ic}), x'(\Jb{\ic})\in (U^-_{n_k})_{\ge 0}
\qquad (1\le \ic\le m)
\end{equation*}
which follows from $x,x'\in U^-_{\ge0}$ and Lemma~\ref{lemm: if and only if lemma which is needed}.
Now, Rietsch's result (Proposition~\ref{prop: Rietsch's result}) implies that we have $x(\Jb{\ic})=x'(\Jb{\ic})$, as claimed in \eqref{eq: x=x' on Jk}. This completes the proof.
\end{proof} 

\vspace{10pt}

\begin{corollary}\label{prop: injection of psi ge 0}
The map $\Psi_{\ge 0} \colon Y_{\ge 0}\to X(\Sigma)_{\ge 0}$ is injective.
\end{corollary}

\begin{proof}
Lemma~\ref{lemma: image of nonnegative part of ZKJ} and the decompositions~\eqref{eq: two decompositions} imply that
\begin{equation*}
\Psi_{\ge 0}(\RS_{K,J;>0})\cap \Psi_{\ge 0}(\RS_{K',J';>0})=\emptyset
\quad \text{for $(K,J)\ne(K',J')$}.
\end{equation*} 
Thus, the claim follows from the previous proposition.
\end{proof}

\vspace{3mm}
To prove that $\Psi_{\ge 0} \colon Y_{\ge 0}\to X(\fan)_{\ge 0}$ is surjective, we first establish its surjectivity on the largest stratum $\Xp{\emptyset,\dyn}$.

\begin{proposition}\label{lemm: surjective from Z empty I ge 0 to X sigma I empty I ge 0}
The map $\Psi_{\ge 0} \colon \RS_{\emptyset,\dyn;>0}\to \Xp{\emptyset, \dyn}$
is surjective.
\end{proposition}

\begin{proof}
By Proposition~\ref{prop: nonnengative stratum} we have 
\begin{equation*}
\Xp{\emptyset,\dyn}= 
\left\{[x_1,\dots, x_{n-1}; 1,\dots, 1]\in X(\fan) \mid x_i>0\ (i\in \dyn)\right\}.
\end{equation*}
Take an arbitrary element 
\begin{equation*}
p=[x_1,\dots, x_{n-1};1,\dots,1]\in \Xp{\emptyset,\dyn}
\end{equation*}
with $x_i>0$ for $i\in \dyn$. Then by Rietsch's result (Proposition~\ref{prop: Rietsch's result}), there exists a Toeplitz matrix $x\in U^-_{\ge 0}$ such that
\begin{equation*}
(\Delta_{1}(x\dot{w}_0),\ldots,\Delta_{n-1}(x\dot{w}_0))
=(x_1,\dots, x_{n-1}).
\end{equation*}
This and Lemma~\ref{lemm: ad part for xwJB-} imply that we have
$\Psi(x\dot{w}_0 B^-)=[x_1,\dots, x_{n-1};1,\dots,1]=p$.
Therefore, it suffices to show that 
\begin{equation}\label{eq: surj goal}
x\dot{w}_0B^-\in\RS_{\emptyset, \dyn;>0}.
\end{equation}
We prove this in what follows.
Since $x$ is a Toeplitz matrix and we have $x_i\ne0$ for all $i\in \dyn$, Corollary~\ref{coro: a description for ZKJ} implies that
\begin{equation}\label{eq: xw0B^-in Zemptyset I}
x\dot{w}_0B^-\in\RS_{\emptyset, \dyn}.\end{equation}
For $t>0$, let $y(t)\coloneqq\exp(tf)$, where $f$ is the regular nilpotent matrix defined in \eqref{eq: def of f}.
Since we have $x\in U^-_{\ge0}$ and $y(t)\in U^-_{>0}$ (\cite[Proposition 5.9]{lu94}), we have
\begin{equation*}
y(t)x\in U^-_{>0}\quad \text{for $t>0$}
\end{equation*}
by \cite[Sect.\ 2.12]{lu94}. This means that $y(t)xB^+\in (\G{n}/B^+)_{>0}$ for $t>0$, and hence
\begin{equation*}
y(t)x\dot{w}_0B^-\in (\G{n}/B^-)_{>0}\quad \text{for $t>0$}
\end{equation*} 
(see the proof of Lemma~\ref{lem: yw0=u+b}).
Taking $t\rightarrow +0$, we obtain that $x\dot{w}_0B^-\in (\G{n}/B^-)_{\ge 0}$.
Together with (\ref{eq: xw0B^-in Zemptyset I}), we now have 
\begin{equation*}
x\dot{w}_0B^-\in\RS_{\emptyset, \dyn}\cap(\G{n}/B^-)_{\ge 0}=\RS_{\emptyset, \dyn;>0},
\end{equation*}
as claimed in \eqref{eq: surj goal}. Hence, the claim follows.
\end{proof}

\vspace{3mm}
\begin{corollary}\label{prop: surjection of psi ge 0}
The map
$\Psi_{\ge 0} \colon Y_{\ge 0}\to X(\fan)_{\ge 0}$ is surjective.
\end{corollary}

\begin{proof}
By Proposition~\ref{lemm: surjective from Z empty I ge 0 to X sigma I empty I ge 0}, we have $\Psi_{\ge 0}(\RS_{\emptyset,\dyn;>0})=\Xp{\emptyset, \dyn}$ for the largest stratum.
Taking the closure in $X(\fan)_{\ge0}$, we obtain that
\begin{equation}\label{eq: closure of nonnegative part of X sigma I empty II}
\overline{\Psi_{\ge 0}(\RS_{\emptyset,\dyn;>0})}
=\overline{\Xp{\emptyset, \dyn}}
=X(\fan)_{\ge0},
\end{equation}
where the last equality follows from Lemma~\ref{lem: closure}.
By definition, we have $\Psi_{\ge 0}(\RS_{\emptyset,\dyn;>0})\subseteq \Psi_{\ge 0}(Y_{\ge 0})$ which implies that
\begin{equation*}
\overline{\Psi_{\ge 0}(\RS_{\emptyset,\dyn;>0})} \subseteq \overline{\Psi_{\ge 0}(Y_{\ge 0})}.
\end{equation*}
Since $Y_{\ge 0}$ and $X(\fan)_{\ge 0}$ are compact Hausdorff spaces, $\Psi_{\ge 0}(Y_{\ge 0})$ is a closed subspace of $X(\fan)_{\ge 0}$. Thus, we obtain
\begin{equation*}
\overline{\Psi_{\ge 0}(\RS_{\emptyset,\dyn;>0})} \subseteq \Psi_{\ge 0}(Y_{\ge 0}).
\end{equation*}
This and \eqref{eq: closure of nonnegative part of X sigma I empty II} imply that
$X(\fan)_{\ge 0}\subseteq \Psi_{\ge 0}(Y_{\ge 0})$ which means that $\Psi_{\ge 0}$ is surjective.
\end{proof}

\begin{remark} 
{\rm
One can also deduce Corollary~$\ref{prop: surjection of psi ge 0}$ by directly showing that $\Psi_{\ge 0}$ is surjective on each stratum $X(\Sigma)_{K,J;\ge 0}$
as in the proof of Proposition~$\ref{lemm: injection on ZKJ>0}$.
}
\end{remark}

\vspace{10pt}
We now prove the main theorem of this section.

\begin{theorem}\label{theo: main theorem}
The map 
$\Psi_{\ge 0} \colon Y_{\ge 0}\to X(\fan)_{\ge 0}$ is a homeomorphism such that 
\begin{equation*}
\Psi_{\ge0}(\RS_{K,J;> 0})=\Xp{K,J}
\end{equation*}
for $K\subseteq J\subseteq \dyn.$ 
\end{theorem}

\begin{proof}
We know that $\Psi_{\ge 0} \colon Y_{\ge 0}\to X(\fan)_{\ge 0}$ is a continuous bijection by Corollary~\ref{prop: injection of psi ge 0} and Corollary~\ref{prop: surjection of psi ge 0}.
Since $Y_{\ge 0}$ and $X(\fan)_{\ge 0}$ are compact Hausdorff space, it follows that $\Psi_{\ge 0}$ is a homeomorphism.

By Lemma~\ref{lemma: image of nonnegative part of ZKJ}, we have $\Psi_{\ge0}(\RS_{K,J;> 0})\subseteq \Xp{K,J}$ for $K\subseteq J\subseteq \dyn$.
Here, $\RS_{K,J;> 0}$ and $\Xp{K,J}$ are the pieces of disjoint decompositions of $Y_{\ge 0}$ and $X(\fan)_{\ge 0}$, respectively (see \eqref{eq: two decompositions}). Since $\Psi_{\ge 0}$ is a bijection, it follows that the equality $\Psi_{\ge0}(\RS_{K,J;> 0})= \Xp{K,J}$ holds for all $K\subseteq J\subseteq \dyn$.
\end{proof}

\vspace{20pt}

\section{A proof of Rietsch's conjecture on $Y_{\ge0}$}
In this section, we prove that Rietsch's conjecture on $Y_{\ge0}$ holds.

\subsection{Statement of the conjecture}
Let us recall the statement of Rietsch's conjecture on $Y_{\ge0}$. 
Let $K\subseteq J\subseteq \dyn$. Recall from \eqref{eq: face post of cube} that we defined the face $E_{K,J}$ of the standard cube $[0,1]^{n-1}$ as
\begin{equation*}
E_{K,J}=
\left\{(x_1,\ldots,x_{n-1})\in [0,1]^{\yn} \ \left| \ \begin{array}{l} x_i=0\quad\text{if $i\in K$},\\
x_i=1\quad\text{if $i\notin J$}
\end{array}
\right.
\right\}.
\end{equation*}
Its relative interior $\Int(E_{K,J})$ is given by
\begin{equation*}
\Int(E_{K,J})
\coloneqq
\left\{(x_i)\in [0,1]^{\yn} \ \left| \ \begin{array}{l} x_i=0\hspace{12mm}\text{if}\quad i\in K,\\
x_i=1\hspace{12mm}\text{if}\quad i\notin J, \\
0<x_i<1\quad\text{otherwise}
\end{array}
\right.
\right\}.
\end{equation*}
In \cite[Sect.~10.3]{ri06}, Rietsch observed that the decomposition $Y_{\ge 0}=\bigsqcup_{K\subseteq J\subseteq \dyn}\RS_{K,J;>0}$ in \eqref{decomposition of nonnegative part of Peterson} has properties similar to that of the cell decomposition $$[0,1]^{n-1}=\bigsqcup_{K\subseteq J\subseteq \dyn}\Int(E_{K,J}).$$
In fact, she showed that 
$\RS_{K,J;> 0}$ is homeomorphic to a cell $\R^{|J|-|K|}_{>0}$ for $K\subseteq J\subseteq \dyn$.
She also proved that $\RS_{K,>0}(\coloneqq Y_{\ge0}\cap \Omega_{w_K}^{\circ})$ is homeomorphic to $\R^{(n-1)-|K|}_{\ge0}$ for $K\subseteq\dyn$
and that $Y_{\ge0}$ is contractible. Based on these observations, 
Rietsch gave the following conjecture in \cite[Conjecture 10.3]{ri06}.\\

\noindent
\textbf{Rietsch's conjecture on $Y_{\ge0}$}\textbf{.}
There is a homeomorphism 
\begin{equation*}
Y_{\ge 0}\to [0,1]^{\yn}
\end{equation*}
such that $\RS_{K,J;>0}$ is mapped to $\Int(E_{K,J})$ for $K\subseteq J\subseteq \dyn$. \\

In the next subsection, we prove that the conjecture holds by applying the main theorem (Theorem~\ref{theo: main theorem}) of this paper. 
For that purpose, let us recall a few notions for polytopes. More details can be found in \cite[Chapter~1]{bu-pa}.

Let $P$ be a polytope. The \textit{face poset} of $P$
is the partially ordered set of faces of $P$, with respect to inclusion. We denote the face poset of $P$ by $\text{Pos}(P)$.
The following claim is well-known; it can be deduced by using simplicial subdivision
for polytopes and simplicial homeomorphisms (See Lemma 2.8 of \cite{mu2}).

\begin{lemma}\label{lemm: combinatorially equivalent induced from a homeomorphism}
Let $P$ and $P'$ be polytopes. If there is a bijection 
$
\alpha \colon \text{\rm Pos}(P)\to \text{\rm Pos}(P')$
preserving the partial orders, 
then there exists a homeomorphism $f\colon P\to P'$ such that 
\begin{enumerate}
\item
$f(F)=\alpha(F)$\quad\text{for $F\in \text{\rm Pos}(P)$},
\item
$f(\Int(F))=\Int(\alpha(F))$\quad\text{for $F\in \text{\rm Pos}(P)$}.
\end{enumerate}
\end{lemma}

\vspace{3mm}

\subsection{A proof of Rietsch's conjecture on $Y_{\ge0}$}
By Theorem~\ref{theo: main theorem}, the map $\Psi_{\ge 0} \colon Y_{\ge 0}\to X(\fan)_{\ge 0}$ is a homeomorphism such that 
\begin{equation*}
\Psi_{\ge0}(\RS_{K,J;> 0})=\Xp{K,J}
\quad \text{for $K\subseteq J\subseteq \dyn.$}
\end{equation*}
As we saw in Section~\ref{subsec: nonnegative and polytope}, the moment map $\mu \colon X(\fan)_{\ge 0}\to \R^{n-1}$ restricts to a homeomorphism 
$\overline{\mu} \colon X(\fan)_{\ge 0}\to P_{n-1}$ such that 
\begin{equation*}
\overline{\mu}(\Xp{K,J})=\Int(F_{K,J})\quad \text{for $K\subseteq J\subseteq \dyn$}.
\end{equation*}
By \eqref{eq: face post of cube 2} and Proposition~\ref{prop: polytope bijection}, the polytope $P_{n-1}$ and the regular cube $[0,1]^{n-1}$ are combinatorially equivalent. Namely, there is a bijection
\begin{equation*}
\alpha \colon \text{Pos}(P_{n-1})\to \text{Pos}([0,1]^{\yn})
\quad ; \quad 
\FCE{K}{J}\mapsto E_{K,J}
\end{equation*}
which preserves the partial orders.
Thus, by Lemma~\ref{lemm: combinatorially equivalent induced from a homeomorphism}, there exists a homeomorphism $f \colon P_{n-1}\to [0,1]^{\yn}$ such that
\begin{equation*}
f(\Int(F_{K,J}))=\Int(E_{K,J})\quad \text{for $K\subseteq J\subseteq \dyn$}.
\end{equation*}
Writing $\varphi\coloneqq f\circ\overline{\mu}\circ\Psi_{\ge 0}$, 
we conclude that the map
\begin{equation*}
\varphi \colon Y_{\ge 0}\to [0,1]^{\yn}
\end{equation*} 
is a homeomorphism such that 
\begin{equation*}
\varphi(\RS_{K,J;>0})=\Int(E_{K,J})
\quad \text{for $K\subseteq J\subseteq \dyn$}.
\end{equation*}
This proves Rietsch's conjecture on $Y_{\ge 0}$.

\end{document}